\documentclass[12pt]{amsart}

\usepackage{amsthm,amsfonts,amsmath,amssymb,latexsym,epsfig,mathrsfs,yfonts,marvosym}
\usepackage[usenames]{color}
\usepackage{graphicx}
\usepackage[all]{xy}
\usepackage{epsfig}
\usepackage{epic}
\usepackage{eepic}
\usepackage{hyperref}
\usepackage{dsfont}\let\mathbb\mathds
\usepackage[english]{babel}

\DeclareMathAlphabet\oldmathcal{OMS}        {cmsy}{b}{n}
\SetMathAlphabet    \oldmathcal{normal}{OMS}{cmsy}{m}{n}
\DeclareMathAlphabet\oldmathbcal{OMS}       {cmsy}{b}{n}
 
\usepackage{eucal}

\usepackage[active]{srcltx}

\newtheorem{theorem}{Theorem}[section]
\newtheorem{lemma}[theorem]{Lemma}
\newtheorem{proposition}[theorem]{Proposition}
\newtheorem{corollary}[theorem]{Corollary}

\newtheorem{def/prop}[theorem]{Definition/Proposition}
\newtheorem{algorithm}[theorem]{Algorithm}

\newenvironment{example}{\medskip \refstepcounter{theorem}
\noindent  {\bf Example \thetheorem}.\rm}{\,}
\newenvironment{remark}{\medskip \refstepcounter{theorem}
\newcommand     {\comment}[1]   {}
\newcommand{\mute}[2] {}
\newcommand     {\printname}[1] {}

\noindent  {\bf Remark \thetheorem}.\rm}{\,}
\newtheorem*{ack}{Acknowledgements}

\def\<{\langle}

\def\>{\rangle}

\def\BOne{{\mathchoice {\rm 1\mskip-4mu l} {\rm 1\mskip-4mu l}
                          {\rm 1\mskip-4.5mu l} {\rm 1\mskip-5mu l}}}

\def\tr{{\rm tr}~}

\def\fract#1#2{\raise4pt\hbox{$ #1 \atop #2 $}}
\def\decdnar#1{\phantom{\hbox{$\scriptstyle{#1}$}}
\left\downarrow\vbox{\vskip15pt\hbox{$\scriptstyle{#1}$}}\right.}

\def\bbc{{\mathbb C}}

\def\bbp{{\mathbb P}}
\def\bbq{{\mathbb Q}}
\def\bbr{{\mathbb R}}

\def\bbz{{\mathbb Z}}

\def\gra{\alpha}
\def\grb{\beta}

\def\grg{\gamma}

\def\grk{\kappa}
\def\grl{\lambda}

\def\gro{\omega}

\def\grs{\sigma}
\def\grt{\tau}

\def\grD{\Delta}

\def\grG{\Gamma}

\def\grL{\Lambda}
\def\grO{\Omega}

\def\bfl{{\bf l}}

\def\bfp{{\bf p}}

\def\bfv{{\bf v}}
\def\bfw{{\bf w}}

\def\cala{{\mathcal A}}

\def\cald{{\mathcal D}}

\def\calf{{\mathcal F}}

\def\cali{{\mathcal I}}

\def\call{{\mathcal L}}

\def\calp{{\mathcal P}}

\def\cals{{\oldmathcal S}}

\def\la#1{\hbox to #1pc{\leftarrowfill}}
\def\ra#1{\hbox to #1pc{\rightarrowfill}}
\def\calz{{\oldmathcal Z}}

\def\gp{{\mathfrak p}}

\def\gt{{\mathfrak t}}
\def\gu{{\mathfrak u}}

\def\gz{{\mathfrak z}}
\def\gtz{\tilde{\mathfrak z}}

\def\gA{{\mathfrak A}}

\def\gI{{\mathfrak I}}

\def\hook{\mathbin{\hbox to 6pt{%
                 \vrule height0.4pt width5pt depth0pt
                 \kern-.4pt
                 \vrule height6pt width0.4pt depth0pt\hss}}}

\def\th{\tilde{h}}
\def\tU{\tilde{U}}

\def\12{\xi_{k_1,k_2}}
\def\m5{M^5_{k_1,k_2}}

\begin{document}

\title{The Sasaki Join, Hamiltonian 2-forms, and Sasaki-Einstein Metrics}

\author{Charles P. Boyer and Christina W. T{\o}nnesen-Friedman}\thanks{Both authors were partially supported by grants from the Simons Foundation, CPB by (\#245002) and CWT-F by (\#208799)}
\address{Charles P. Boyer, Department of Mathematics and Statistics,
University of New Mexico, Albuquerque, NM 87131.}
\email{cboyer@math.unm.edu} 
\address{Christina W. T{\o}nnesen-Friedman, Department of Mathematics, Union
College, Schenectady, New York 12308, USA } \email{tonnesec@union.edu}

\keywords{Sasaki-Einstein metrics, join construction, Hamiltonian 2-form}

\subjclass[2000]{Primary: 53D42; Secondary:  53C25}

\maketitle

\markboth{Sasaki Join, K\"ahler Admissible}{Charles P. Boyer and Christina W. T{\o}nnesen-Friedman}


\begin{abstract}
By combining the join construction \cite{BGO06} from Sasakian geometry with the Hamiltonian 2-form construction \cite{ApCaGa06} from K\"ahler geometry, we recover Sasaki-Einstein metrics discovered by physicists \cite{GMSW04b}. Our geometrical approach allows us to give an algorithm for computing the topology of these Sasaki-Einstein manifolds. In particular, we explicitly compute the cohomology rings for several cases of interest and give a formula for homotopy equivalence in one particular 7-dimensional case. We also show that our construction gives at least a two dimensional cone of both Sasaki-Ricci solitons and extremal Sasaki metrics.
\end{abstract}


\section{Introduction}

Recently the authors have been able to obtain many new results on extremal Sasakian geometry \cite{BoTo12b,BoTo11,BoTo13} by giving a geometric construction that combines the `join construction' of \cite{BG00a,BGO06} with the `admissible construction of Hamiltonian 2-forms' for extremal K\"ahler metrics described in \cite{ApCaGa06,ACGT04,ACGT08,ACGT08c}. In particular, we have proven the existence of Sasaki metrics of constant scalar curvature in a countable number of contact structures of Sasaki type with 2-dimensional Sasaki cones on $S^3$-bundles over Riemann surfaces of positive genus. In this case the contact bundle $\cald$ satisfies $c_1(\cald)\neq 0$. These are all the 2-dimensional Sasaki cones on $S^3$-bundles over Riemann surfaces that one can construct through our method. In the present paper we consider the case of certain lens space bundles over a compact positive K\"ahler-Einstein manifold whose contact bundle has vanishing first Chern class. In this case a constant scalar curvature Sasaki metric (CSC) implies the existence of a Sasaki-Einstein metric in the CSC ray.

Earlier, physicists \cite{GHP03,GMSW04a,GMSW04b,CLPP05,MaSp05b} working on the AdS/CFT correspondence had discovered a method of constructing Sasaki-Einstein metrics in dimension $2n+3$ from a K\"ahler-Einstein metric on a $2n$-dimensional manifold. Their method is very closely related to the Hamiltonian 2-form approach of \cite{ApCaGa06} (cf. Section 4.3 of \cite{Spa10}). In the present paper we show that the physicist's results fit naturally into our geometric construction. Furthermore, our geometric approach leads naturally to an algorithm for computing the cohomology ring of the $2n+3$-manifolds. In particular, we explicitly compute the cohomology ring of all such examples in dimension $7$ showing that there are a countably infinite number of distinct homotopy types of such manifolds.

Our procedure involves taking a join of a regular Sasaki-Einstein manifold $M$ with the weighted 3-sphere $S^3_\bfw$, that is, $S^3$ with its the standard contact structure, but with a weighted contact 1-form whose Reeb vector field generates rotations with different weights for the two complex coordinates $z_1,z_2$ of $S^3\subset \bbc^2$. We call this the $S^3_\bfw$-join.

\begin{theorem}\label{admjoinse}
Let $M_{l_1,l_2,\bfw}=M\star_{l_1,l_2}S^3_\bfw$ be the $S^3_\bfw$-join with a regular Sasaki manifold $M$ which is an $S^1$-bundle over a compact positive K\"ahler-Einstein manifold $N$ with a primitive K\"ahler class $[\gro_N]\in H^2(N,\bbz)$. Assume that the relatively prime positive integers $(l_1,l_2)$ are the relative Fano indices given explicitly by 
$$l_{1}=\frac{\cali_N}{\gcd(w_1+w_2,\cali_N)},\qquad   l_2=\frac{w_1+w_2}{\gcd(w_1+w_2,\cali_N)},$$ where $\cali_N$ denotes the Fano index of $N$. 

Then for each vector $\bfw=(w_1,w_2)\in \bbz^+\times\bbz^+$ with relatively prime components satisfying $w_1>w_2$ there exists a Reeb vector field $\xi_\bfv$ in the 2-dimensional $\bfw$-Sasaki cone on $M_{l_1,l_2,\bfw}$ such that the corresponding Sasakian structure $\cals=(\xi_\bfv,\eta_\bfv,\Phi,g)$ is Sasaki-Einstein.  Additionally (up to isotopy) the Sasakian structure associated to every single ray, $\xi_\bfv$, in the $\bfw$-Sasaki cone is a Sasaki-Ricci soliton as well as extremal.
\end{theorem}

By the $\bfw$-Sasaki cone we mean the subcone of Sasaki cone spanned by the 2-vector $\bfw$.
Note  also that if $N$ has no Hamiltonian vector fields, then the $\bfw$-Sasaki cone is the full Sasaki cone, so in this case the Sasaki cone is exhausted by extremal Sasaki metrics as well as Sasaki-Ricci solitons.

Most of the SE structures in Theorem \ref{admjoinse} are irregular. Such structures have irreducible transverse holonomy \cite{HeSu12b}, implying there can be no generalization of the join procedure to the irregular case. We must deform within the Sasaki cone to obtain them. Furthermore, it follows from \cite{RoTh11,CoSz12} that constant scalar curvature Sasaki metrics (hence, SE) imply a certain K-semistability.

An easy case in all possible dimension is joining with the standard odd dimensional sphere in which case we obtain:

\begin{theorem}\label{setop}
For each pair of ordered ($w_1>w_2$) relatively prime positive integers $(w_1,w_2)$ there is a $2r+3$-dimensional Sasaki-Einstein manifold $M^{2r+3}_{l_1,l_2,\bfw}$ whose cohomology ring is 
$$\bbz[x,y]/(w_1w_2l_1^2x^2,x^{r+1},x^2y,y^2)$$
where $x,y$ are classes of degree $2$ and $2r+1$, respectively, and $(l_1,l_2)$ are given by
$$(l_1,l_2)=\bigl(\frac{r+1}{\gcd(w_1+w_2,r+1)},\frac{w_1+w_2}{\gcd(w_1+w_2,r+1)}\bigr).$$
\end{theorem}

This theorem provides us with many examples of Sasaki-Einstein manifolds with isomorphic cohomology rings. We have 
\begin{corollary}\label{cor1}
Let $k$ denote the length of the prime decomposition of $w_1w_2$, then there are $2^{k-1}$ simply connected Sasaki-Einstein manifolds of dimension $2r+3$ determined by Theorem \ref{setop} with isomorphic cohomology rings such that $H^4$ has order $w_1w_2l_1^2$.
\end{corollary}

For the manifolds $M^{7}_{l_1,l_2,\bfw}$ of dimension 7 ($r=2$) much more is known about the topology. These are special cases of what are called generalized Witten spaces in \cite{Esc05}. In particular, the homotopy type was given in \cite{Kru97}, while the homeomorphism and diffeomorphism type was given in \cite{Esc05}. For our subclass admitting Sasaki-Einstein metrics we give necessary and sufficient conditions on $\bfw$ for homotopy equivalence when the order of $H^4$ is odd in Proposition \ref{homequivprop} below. Thus, we answer in the affirmative the existence of Einstein metrics on certain generalized Witten manifolds.

Aside from the $\bfw$-join with the 5-sphere, for dimension 7 our method produces Sasaki-Einstein manifolds $M_{k,\bfw}^7$ on lens space bundles over all del Pezzo surfaces $\bbc\bbp^2\# k\overline{\bbc\bbp}^2$ that admit a K\"ahler-Einstein metric. In particular 

\begin{theorem}\label{delPezzo}
For each relatively prime pair $(w_1,w_2)$ there exist Sasaki-Einstein metrics on the 7-manifolds $M^7_{k,\bfw}$ with cohomology ring
$$H^q(M^7_{k,\bfw},\bbz)\approx \begin{cases}
                    \bbz & \text{if $q=0,7$;} \\
                    \bbz^{k+1} & \text{if $q=2,5$;} \\
                    \bbz^k_{w_1+w_2}\times \bbz_{w_1w_2} & \text{if $q=4;$} \\
                     0 & \text{if otherwise}, 
                     \end{cases}$$
with the ring relations determined by $\gra_i\cup \gra_j=0, w_1w_2s^2=0,\break (w_1+w_2)\gra_i\cup
s=0,$ and $\gra_i,s$ are the $k+1$ two classes with $i=1,\cdots k$ where $k=3,\cdots,8$. Furthermore, when $4\leq k\leq 8$ the local moduli space of Sasaki-Einstein metrics has real dimension $4(k-4)$.
\end{theorem}

The paper is organized as follows. Section 2 gives a brief review of ruled manifolds that are the projectivization of a complex rank 2 vector bundle of the form $S=\bbp(\BOne\oplus L)$ over a K\"ahler-Einstein manifold $N$. These admit Hamiltonian 2-forms that give rise to the K\"ahler admissible construction that is necessary for our procedure. In the somewhat long Section 3 we describe our join construction, in particular, the join with the weighted 3-sphere, $S^3_\bfw$. We then discuss in detail the orbit structure of quasi-regular Reeb vector fields in the $\bfw$-Sasaki cone. Generally, the quotients appear as orbifold log pairs $(S,\grD)$ which fiber over $N$ with fiber an orbifold of the form $\bbc\bbp^1[v_1,v_2]/\bbz_m$, where $\grD$ is branch divisor, and $\bbc\bbp^1[v_1,v_2]$ is a weighted projective space. Moreover, we give conditions for the existence of a regular Reeb vector field, that is, when the quotient has a trivial orbifold structure. In Section 4 we discuss the topology of the joins, giving an algorithm for computing the integral cohomology ring. We work out the details in several cases which gives the proofs of Theorems \ref{setop}, \ref{delPezzo}, and Corollary \ref{cor1}. Finally, in Section 5 we present the details of the admissibility conditions that proves Theorem \ref{admjoinse}.

\begin{ack}
The authors would like to thank David Calderbank for discussions and Matthias Kreck for providing us with a proof that the first Pontrjagin class is a homeomorphism invariant.
\end{ack}

\section{Ruled Manifolds}
In this section we consider ruled manifolds of the following form. Let $(N,\gro_N)$ be a compact K\"ahler manifold with primitive integer K\"ahler class $[\gro_N]$, that is, a Hodge manifold. Consider a rank two complex vector bundle of the form $E=\BOne\oplus L$ where $L$ is a complex line bundle on $N$ and $\BOne$ denotes the trivial bundle. By a ruled manifold we shall mean the projectivization $S=\bbp(\BOne\oplus L)$. We can view $S$ as a compactification of the complex line bundle $L$ on $N$ by adding the `section at infinity'. For $x\in N$ we let $(u,v)$ denote a point of the fiber $E_x=\BOne\oplus L_x$. There is a natural action of $\bbc^*$ (hence, $S^1$) on $E$ given by $(c,z)\mapsto (c,\grl z)$ with $\grl\in \bbc^*$. The action $z\mapsto \grl z$ is a complex irreducible representation of $\bbc^*$ determined by the line bundle $L$. Such representations (characters) are labeled by the integers $\bbz$. Thus, we write $L=L_n$ for $n\in \bbz$ and refer to $n$ as the `degree' of $L$.

\subsection{A Construction of Ruled Manifolds}\label{ruledsec}
We now give a construction of such manifolds. Let $S^1\ra{1.6} M\ra{1.6} N$ be the circle bundle over $N$ determined by the class $[\gro_N]\in H^2(N,\bbz)$. We denote the $S^1$-action by $(x,u)\mapsto (x,e^{i\theta}u)$. Now represent $S^3\subset \bbc^2$ as $|z_1|^2+|z_2|^2=1$ and consider an $S^1$-action on $M\times S^3$ given by $(x,u;z_1,z_2)\mapsto (x,e^{i\theta}u;z_1,e^{in\theta}z_2)$. There is also the standard $S^1$-action on $S^3$ given by $(z_1,z_2)\mapsto (e^{i\chi}z_1,e^{i\chi}z_2)$ giving a $T^2$-action on $M\times S^3$ defined by
\begin{equation}\label{T2act}
(x,u;z_1,z_2)\mapsto (x,e^{i\theta}u;e^{i\chi}z_1,e^{i(\chi+n\theta)}z_2).
\end{equation}

\begin{lemma}\label{T2quot}
The quotient by the $T^2$-action of Equation (\ref{T2act}) is the projectivization $S_n=\bbp(\BOne\oplus L_n)$.
\end{lemma}

\begin{proof}
First we see from (\ref{T2act}) that the action is free, so there is a natural bundle projection $(M\times S^3)/T^2\ra{1.6} N$ defined by $\pi(x,[u;z_1,z_2])=x$ where the bracket denotes the $T^2$ equivalence class. The fiber is $\pi^{-1}(x)=[u;z_1,z_2]$ which since $u$ parameterizes a circle is identified with $S^3/S^1=\bbc\bbp^1$. This bundle is trivial if and only if $n=0$ and $n$ labels the irreducible representation of $S^1$ on the line bundle $L_n$.
\end{proof}

We can take the line bundle $L_1$ to be any primitive line bundle in ${\rm Pic}(N)$. In particular, we are interested in the taking $L_1$ to be the line bundle associated to the primitive cohomology class $[\gro_N]\in H^2(N,\bbz)$. Then we have

\begin{lemma}\label{c1L}
The following relation holds: $c_1(L_n)=n[\gro_N]$.
\end{lemma}

\begin{proof}
Equation (\ref{T2act}) implies that the $S^1$-action on the line bundle $L_n$ is given by $z\mapsto e^{in\theta}z$. But we know that the definition of $M$ that it is the unit sphere in the line bundle over $N$ corresponding to $n=1$, and this corresponds to the class $[\gro_N]$, that is $c_1(L_1)=[\gro_N]$. Thus, $c_1(L_n)=n[\gro_N]$.
\end{proof}

\subsection{The Admissible Construction}\label{adm-basics}
Let us now suppose that $(N,\gro_N$) is K\"ahler-Einstein with K\"ahler metric $g_N$ and Ricci form 
$\rho_N = 2\pi\cali_N \omega_N$, where $\cali_N$ denotes the Fano index.
We will also assume that $n$ from Section \ref{ruledsec} is non-zero. Then 
$(\omega_{N_n},g_{N_n}): =(2n\pi \omega_N, 2n\pi g_N)$ satisfies that 
$( g_{N_n}, 
\omega_{N_n})$ or $(- g_{N_n}, 
-\omega_{N_n})$
is a K\"ahler structure (depending on the sign of $n$). 
In either case, we let $(\pm g_{N_n}, \pm \omega_{N_n})$ refer to the K\"ahler structure.
We denote the real dimension of $N$ by
$2 d_{N}$ and write the scalar curvature of $\pm g_{N_n}$ as $\pm 2 d_{N_n} s_{N_n}$.
[So, if e.g. $-g_{N_n}$ is a K\"ahler structure with positive scalar curvature, $s_{N_n}$ would be negative.] Since the (scale invariant) Ricci form is given by $\rho_N=s_{N_n}\omega_{N_n}$, it is easy to see that $s_{N_n}=\cali_N/n$.

Now Lemma \ref{c1L} implies that
$c_{1}(L_n)= [\omega_{N_n}/2\pi]$.
Then, following \cite{ACGT08}, the total space of the projectivization
$S_n=\bbp(\BOne\oplus L_n)$ is called {\it admissible}.

On these manifolds, a particular type of K\"ahler metric on $S_n$,  also called
{\it admissible}, can now be constructed \cite{ACGT08}. We shall describe this construction in Section \ref{KE} where we will use it to prove Theorem \ref{admjoinse}.

An admissible K\"ahler manifold is a special case of a K\"ahler manifold admitting a so-called Hamiltonian $2$-form.
Let $(S,J,\gro,g)$ be a K\"ahler manifold of real dimension $2n$. Recall \cite{ApCaGa06} that on $(S,J,\gro,g)$ a {\it Hamiltonian 2-form} is a $J$-invariant 2-form $\phi$ that satisfies the differential equation
\begin{equation}\label{ham2form}
2\nabla_X\phi = d\tr\phi\wedge (JX)^\flat-d^c\tr\phi\wedge X^\flat
\end{equation}
for any vector field $X$. Here $X^\flat$ indicates the 1-form dual to $X$, and $\tr\phi$ is the trace with respect to the K\"ahler form $\gro$, i.e. $\tr\phi=g(\phi,\gro)$ where $g$ is the K\"ahler metric.
Note that if $\phi$ is a Hamiltonian $2$-form, then so is $\phi_{a,b} = a\phi + b\omega$ for any constants $a,b \in \bbr$.

A Hamiltonian $2$-form $\phi$ induces an isometric hamiltonian $l$-torus action on $S$ for some $0\leq l \leq n$. 
To see this, we follow \cite{ApCaGa06} and use the K\"ahler form
 $\omega$ to identify $\phi$ with a Hermitian endomorphism.  Consider now the elementary symmetric functions $\grs_1,\ldots,\grs_n$ of its $n$ eigenvalues. The Hamiltonian vector fields $K_i=J{\rm grad}\grs_i$ are Killing with respect to the K\"ahler 
 metric $g$. Moreover the Poisson brackets $\{\sigma_i,\sigma_j\}$ all vanish and so, in particular the
 vector fields $K_1,...,K_n$ commute. In the case where $K_1,...,K_n$ are independent, we have that $(S,J,\gro,g)$ is toric. In fact, it is a very special kind of toric, namely {\em orthotoric} \cite{ApCaGa06}. In general, it is proved in  \cite{ApCaGa06} that there exists a number
 $0 \leq l \leq n$ such that  the span of $K_1,...,K_n$ is everywhere at most $l-dimensional$ and on an open dense set $S^0$
 $K_1,...,K_l$ are linearly independent.
 Now $l$ is called the {\it order} of $\phi$. We have $0\leq l \leq n$. 
 
The admissible metrics as described in section \ref{KE} admit a Hamiltonian $2$-form of order one (see Remark \ref{hamiltonian2form}).

\section{The Join Construction}
The join construction was first introduced in \cite{BG00a} for Sasaki-Einstein manifolds, and later generalized to any quasi-regular Sasakian manifolds in \cite{BGO06} (see also Section 7.6.2 of \cite{BG05}). However, as pointed out in \cite{BoTo11} it is actually a construction involving the orbifold Boothby-Wang construction \cite{BoWa,BG00a}, and so applies to quasi-regular strict contact structures. Although it is quite natural to do so, we do not need to fix the transverse (almost) complex structure. Moreover, in \cite{BoTo13} it was shown that in the special case of $S^3$-bundles over Riemann surfaces a twisted transverse complex structure on a regular Sasakian manifold can be realized by a product transverse complex structure on a certain quasi-regular Sasakian structure in the same Sasaki cone. We are interested in exploring just how far this analogy generalizes.

The general join construction proceeds as follows. Given compact quasi-regular contact manifolds $(M_i,\eta_i)$ where $\cals_i$ is a Sasakian structure with contact forms $\eta_i$ and $\dim M_i=2n_i+1$ together with a pair of relatively prime integers $(l_1,l_2)$, we construct new contact orbifolds $M_1\star_{l_1,l_2}M_2$ of dimension $2(n_1+n_2)+1$ as follows: the Reeb vector fields $\xi_i$ of $\eta_i$ generate a locally free circle action whose quotient $\calz_i$ is a symplectic orbifold with symplectic form $\gro_i$ satisfying $\pi_i^*\gro_i=d\eta_i$ where $\pi_i:M_i\ra{1.6} \calz_i$ is the natural quotient projection. The product orbifold $\calz_1\times \calz_2$ has symplectic structures $\gro_{l_1,l_2}=l_1\gro_1+l_2\gro_2$ where the forms $\gro_i$ define primitive cohomology classes in $H^2_{orb}(\calz_i,\bbz)$. This implies that the cohomology class $[\gro_{l_1,l_2}]$ is also primitive. The orbifold Boothby-Wang construction \cite{BG00a} then gives a contact structure with contact form $\eta_{l_1,l_2}$ on the total space $M_1\star_{l_1,l_2}M_2$ of the $S^1$-orbibundle over $\calz_1\times\calz_2$ which satisfies $d\eta_{l_1,l_2}=\pi^*\gro_{l_1,l_2}$ where $\pi:M_1\star_{l_1,l_2}M_2\ra{1.6} \calz_1\times\calz_2$ is the natural orbibundle projection. We also note that $\eta_{l_1,l_2}$ is a connection 1-form in this orbibundle. The orbifold $M_1\star_{l_1,l_2}M_2$ can also be given as the quotient space of a locally free circle action on the manifold $M_1\times M_2$. Let $\xi_i$ denote the Reeb vector field of the contact manifold $(M_i,\eta_i)$. We consider $\eta_{l_1,l_2}=l_1\eta_1+l_2\eta_2$ as a 1-form on the product $M_1\times M_2$. From the Reeb vector fields $\xi_i$ we construct the vector field  $L_{l_1,l_2}=\frac{1}{2l_1}\xi_1-\frac{1}{2l_2}\xi_2$ which induces a locally free circle action on $M_1\times M_2$. This action is free if $\gcd(l_2\upsilon_1,l_1\upsilon_2)=1$ in which case the quotient is a smooth manifold. It is called the {\it $(l_1,l_2)$-join} of $M_1$ and $M_2$ and denoted by $M_1\star_{l_1,l_2}M_2$. Here $\upsilon_i$ denotes the {\it order} of $M_i$, that is, the lcm of the orders of the isotropy groups of the locally free action.

So far nothing has been said about Sasakian or K\"ahlerian structures. However, given any symplectic structure, one can choose a compatible almost complex structure, giving an almost K\"ahler manifold. Similarly, given a quasi-regular contact structure one can choose a compatible transverse almost complex structure giving a K-contact manifold. Furthermore, the transverse almost complex structure on $M$ is the horizontal lift of the almost complex structure on $N$, and when these structures are integrable $M$ is Sasakian and $N$ is K\"ahlerian. We have the following easily verified facts about the Sasaki cone and the joins of Sasaki manifolds:

\begin{itemize}
\item The join of extremal Sasaki metrics gives an extremal Sasaki metric.
\item The join of CSC Sasaki metrics gives a CSC Sasaki metric.
\item Let $(M_i,\cals_i)$ be quasi-regular Sasakian manifolds with Sasaki cones $\grk_i$, respectively. Let $\grk_\star$ denote the Sasaki cone of the join $\cals_1\star_{l_1,l_2}\cals_2$. Then we have $\dim\grk_\star=\dim\grk_1+\dim\grk_2-1$.
\item If the dimension of the Sasaki cone $\grk(\cals)$ is greater than $1$ then $\cals$ is of positive or indefinite type.
\item If $\dim\gA\gu\gt(\cals)>1$, then $\dim\grk(\cals)>1$.
\end{itemize}

\subsection{The $S^3_\bfw$-Join}
Here we apply the $(l_1,l_2)$-join construction to the case at hand, namely, the $(l_1,l_2)$-join of a regular Sasaki manifold $M$ with the weighted 3-sphere $S^3_\bfw$. $M$ is a the total space of an $S^1$-bundle over a smooth projective algebraic variety $N$ with K\"ahler form $\gro_N$. The Reeb vector field $\xi_2$ of the weighted sphere $S^3_\bfw$ is given by $\xi_\bfw=\sum_{i=1}^2w_iH_i$ where $H_i$ is the vector field on $S^3$ induced by $y_i\partial_{x_i}-x_i\partial_{y_i}$ on $\bbr^4$, and $w_1,w_2$ are relatively prime positive integers.  In this case the smoothness condition becomes $\gcd(l_2,l_1w_1w_2)=1$ which, since $gcd(l_1,l_2)=1$, is equivalent to $\gcd(l_2,w_i)=1$ for $i=1,2$.

We consider the quotient $M\star_{l_1,l_2}S^3_\bfw$ of $M\times S^3$ by the $S^1$-action generated by the vector field $L_{l_1,l_2}$.
Moreover, the 1-form $\eta_{l_1,l_2}$ passes to the quotient and gives it a contact structure. The Reeb vector field of $\eta_{l_1,l_2}$ is the vector field
\begin{equation}\label{Reebjoin}
\xi_{l_1,l_2}=\frac{1}{2l_1}\xi_1+\frac{1}{2l_2}\xi_2.
\end{equation}
and we have the commutative diagram
\begin{equation}\label{s2comdia}
\begin{matrix}  M\times S^3 &&& \\
                          &\searrow\pi_L && \\
                          \decdnar{\pi_{12}} && M\star_{l_1,l_2}S^3_\bfw &\\
                          &\swarrow\pi && \\
                        N\times\bbc\bbp^1[\bfw] &&& 
\end{matrix}
\end{equation}
where the $\pi$s are the obvious projections, and $\bbc\bbp^1[\bfw]$ is the weighted projective space usually written as $\bbc\bbp(w_1,w_2)$.

\subsection{The Cohomological Einstein Condition}
Let us compute the first Chern class of our induced contact structure $\cald_{l_1,l_2,\bfw}$ on $M\star_{l_1,l_2}S^{3}_\bfw$. For this we compute the orbifold first Chern class of $N\times \bbc\bbp^1[\bfw]$, viz. 
\begin{equation}\label{c1N}
c_1^{orb}(N\times \bbc\bbp^1[\bfw])=c_1(N)+\frac{|\bfw|}{w_1w_2}PD(E)
\end{equation}
as an element of $H^2(N\times \bbc\bbp^1[\bfw],\bbq)\approx H^2(N,\bbq)\oplus H^2(\bbc\bbp^1[\bfw]),\bbq)$. Now the K\"ahler form on $N\times \bbc\bbp^1[\bfw]$ is $\gro_{l_1,l_2}=l_1\gro_N+ l_2\gro_\bfw$ where $\gro_\bfw$ is the standard K\"ahler form on $\bbc\bbp^1[\bfw]$ such that $p^*[\gro_\bfw]$ is a positive generator in $H^2_{orb}(\bbc\bbp^1[\bfw],\bbz)$ where $p:\mathsf{B}\bbc\bbp^1[\bfw]\ra{1.6} \bbc\bbp^1[\bfw]$ is the natural orbifold classifying projection, that is, $\mathsf{B}\bbc\bbp^1[\bfw]$ is the classfying space associated to any groupoid representing the orbifold $ \bbc\bbp^1[\bfw]$ \cite{Hae84}. Pulling $\gro_{l_1,l_2}$ back to the join $M_{l_1,l_2,\bfw}= M\star_{l_1,l_2}S^{3}_\bfw$ we have $\pi^*\gro_{l_1,l_2}=d\eta_{l_1,l_2,\bfw}$ implying that $l_1\pi^*[\gro_N]+l_2\pi^*[\gro_\bfw]=0$ in $H^2(M_{l_1,l_2,\bfw},\bbz)$. So letting $\gI$ denote the ideal generated by $l_1\pi^*[\gro_N]+l_2\pi^*[\gro_\bfw]$ we have
\begin{equation}\label{c1cald}
c_1(\cald_{l_1,l_2,\bfw})=\bigl(\pi^*c_1(N)+|\bfw|\pi^*[\gro_\bfw]\bigr)/\gI.
\end{equation}

Now we assume that $N$ is K\"ahler-Einstein with K\"ahler form $\gro_N$ which for convenience we assume to define a primitive class $[\gro_N]\in H^2(N,\bbz)$. We also assume that $(N,\gro_N)$ is monotone, that is, that $c_1(N)=\cali_N[\gro_N]$ where $\cali_N$ is the Fano index of $N$. So Equation (\ref{c1cald}) becomes
\begin{equation}\label{c1cald2}
c_1(\cald_{l_1,l_2,\bfw})=\bigl(\cali_N\pi^*[\gro_N]+|\bfw|\pi^*[\gro_\bfw]\bigr)/\gI.
\end{equation}
The cohomological Einstein condition is $c_1(\cald_{l_1,l_2,\bfw})=0$ or equivalently that $\bigl(\cali_N\pi^*[\gro_N]+|\bfw|\pi^*[\gro_\bfw]\bigr)$ lies in the ideal $\gI$. This implies the condition $l_2\cali_N=|\bfw|l_{1}$. We have arrived at:

\begin{lemma}\label{c10}
Necessary conditions for the Sasaki manifold $M_{l_1,l_2,\bfw}$ to admit a Sasaki-Einstein metric is that $\cali_N>0$, and that  
$$l_2=\frac{|\bfw|}{\gcd(|\bfw|,\cali_N)},\qquad l_{1}=\frac{\cali_N}{\gcd(|\bfw|,\cali_N)}.$$
\end{lemma}

\begin{remark}\label{relindex}
The integers $l_1,l_2$ in Lemma \ref{c10} were called {\it relative Fano indices} in \cite{BG00a}. For the remainder of the paper we assume that these integers take the values given by Lemma \ref{c10}.
\end{remark}

\subsection{The Tori Actions}
Consider the action of the $3$-dimensional torus $T^{3}$ on the product $M\times S^{3}_\bfw$ defined by
\begin{equation}\label{Taction}
(x,u;z_1,z_2)\mapsto (x,e^{il_2\theta}u;e^{i(\phi_1-l_1w_1\theta)}z_1,e^{i(\phi_2-l_1w_2\theta)}z_2).
\end{equation}
The Lie algebra $\gt_{3}$ of $T^{3}$ is generated by the vector fields $L_{l_1,l_2,\bfw},H_1,H_2$. 

Our join manifold $M_{l_1,l_2,\bfw}=M\star_{l_1,l_2}S^{3}_\bfw$ is the quotient by the circle subgroup $S^1_\theta$ obtained by setting $\phi_i=0$ for $i=1,2$ in Equation (\ref{Taction}). We can realize the quotient procedure in two stages. First, divide by the cyclic subgroup $\bbz_{l_2}\subset S^1_\theta$ to get $M\times L(l_2;l_1\bfw)$ and then divide by the quotient $S^1_\theta/\bbz_{l_2}$ to realize $M_{l_1,l_2,\bfw}$ as an associated fiber bundle to the principal $S^1$-bundle over $N$ with fiber $L(l_2;l_1\bfw)$. Here $L(l_2;l_1\bfw)$ is the lens space $L(l_2;l_1w_1,l_1w_2)$.

The infinitesimal generator of the $S^1_\theta$ action is given by
\begin{equation}\label{infq+2action}
L_{l_1,l_2}=\frac{1}{2l_1}\xi_1-\sum_{j=1}^{2}\frac{1}{2l_2}w_jH_j,
\end{equation}
and we denote its Lie subalgebra by $\{L_{l_1,l_2},\bfw\}$ to indicate its dependence on the weight vector $\bfw$. We have an exact sequence of Abelian Lie algebras
\begin{equation}\label{Liealgexact}
0\ra{2.5}\{L_{l_1,l_2,\bfw}\}\ra{2.5} \gt_{3}\ra{2.5} \gt_{2}\ra{2.5} 0.
\end{equation}
The quotient algebra $\gt_{2}$ is generated by the $H_1,H_2$. So as in \cite{BoTo13} the Sasaki cone $\gt_{2}^+$ on $M_{l_1,l_2,\bfw}$ is inherited by the Sasaki cone on $S^{3}$. We refer to this 2-dimensional cone as the {\it $\bfw$-Sasaki cone} or just {\it $\bfw$-cone}. We remark that $\gt_{2}^+$ may not be the full Sasaki cone, since the Sasakian structure on $M$ may have a Sasaki automorphism group whose dimension is greater than one. However, here we shall only work with the restricted Sasaki cone $\gt_{2}^+$.

Next we consider the 2-torus action generated by setting $\phi_i=v_i\phi$ in Equation (\ref{Taction}) where for now $v_1,v_2$ are relatively prime positive integers. As in the proof of Proposition 3.8 of \cite{BoTo13} we get a commutative diagram
\begin{equation}\label{comdia1}
\begin{matrix}  M\times S^{3}_\bfw &&& \\
                          &\searrow && \\
                          \decdnar{\pi_B} && M_{l_1,l_2,\bfw} &\\
                          &\swarrow && \\
                          B_{l_1,l_2,\bfv,\bfw} &&& 
\end{matrix}
\end{equation}
where $B_{l_1,l_2,\bfv,\bfw}$ is a bundle over $N$ with fiber a weighted projective space, and $\pi_B$ denotes the quotient projection by $T^2$. The Lie algebra of this $T^2$ is generated by $L_{l_1,l_2,\bfw},\sum_jv_jH_j$. 

The $T^2$ action describing the $\pi_B$ quotient of Diagram (\ref{comdia1}) is given in this case, as in Equation (\ref{Taction}), by
\begin{equation}\label{T2action2}
(x,u;z_1,z_2)\mapsto (x,e^{i\frac{|\bfw|}{\gcd(|\bfw|,\cali_N)}\theta}u;e^{i(v_1\phi-\frac{\cali_N}{\gcd(|\bfw|,\cali_N)}w_1\theta)}z_1,e^{i(v_2\phi-\frac{\cali_N}{\gcd(|\bfw|,\cali_N)}w_2\theta)}z_2).
\end{equation}
First divide by the finite subgroup $\bbz_{l_2}$ of $S^1_\theta$ giving $M\times L(l_2;l_1w_1,l_1w_2)$, where $L(l_2;l_1w_1,l_1w_2)$ is the lens space. Then the residual $S^1_\theta/\bbz_{l_2}\approx S^1$ action is 
\begin{equation}\label{resS1act}
(x,u;z_1,z_2)\mapsto (x,e^{i\theta}u;[e^{-i\frac{l_1w_1}{l_2}\theta)}z_1,e^{-i\frac{l_1w_2}{l_2}\theta)}z_2]).
\end{equation}
This describes 
$$M_{l_1,l_2,\bfw}=M\times_{S^1}L(l_2;l_1w_1,l_1w_2)$$ 
as an associated fiber bundle to the principal $S^1$-bundle $M\ra{1.5} N$ with fiber $L(l_2;l_1w_1,l_1w_2)$ over the K\"ahler manifold $N$. The brackets in Equation (\ref{resS1act}) denote the equivalence class defined by $(z'_1,z'_2)\sim (z_1,z_2)$ if $(z'_1,z'_2)=(\grl^{l_1w_1}z_1,\grl^{l_1w_2}z_2)$ for $\grl^{l_2}=1$.

We now turn to the $T^2$ action of $S^1_\phi\times (S^1_\theta/\bbz_{l_2})$ on $M\times L(l_2;l_1w_1,l_1w_2)$ given by
\begin{equation}\label{T2action3}
(x,u;z_1,z_2)\mapsto (x,e^{i\theta}u;[e^{i(v_1\phi-\frac{l_1w_1}{l_2}\theta)}z_1,e^{i(v_2\phi-\frac{l_1w_2}{l_2}\theta)}z_2]),
\end{equation}

 This gives rise to the Diagram (\ref{comdia1})
\begin{equation}\label{comdia2}
\begin{matrix}  M\times L(l_2;l_1w_1,l_1w_2) &&& \\
                          &\searrow && \\
                          \decdnar{\pi_B} && M_{l_1,l_2,\bfw} &\\
                          &\swarrow && \\
                          B_{l_1,l_2,\bfv,\bfw} &&& 
\end{matrix}
\end{equation}

Let us analyze the behavior of the $T^2$ action given by Equation (\ref{T2action3}). We shall see that it it not generally effective. First we notice that the $S^1_\theta$ action is free since it is free on the first factor. Next we look for fixed points under a subgroup of the circle $S^1_\phi$. Thus, we impose 
$$(e^{iv_1\phi}z_1,e^{iv_2\phi}z_2)=(e^{-2\pi\frac{\cali_Nw_1}{|\bfw|}ri}z_1,e^{-2\pi\frac{\cali_Nw_2}{|\bfw|}ri}z_2)$$
for some $r=0,\ldots,\frac{|\bfw|}{\gcd(|\bfw|,\cali_N)}-1$. If $z_1z_2\neq 0$ we must have
\begin{equation}\label{phisoln}
v_1\phi=2\pi(-\frac{\cali_Nw_1r}{|\bfw|} +k_1),\qquad v_2\phi=2\pi(-\frac{\cali_Nw_2r}{|\bfw|} +k_2)
\end{equation}
for some integers $k_1,k_2$ which in turn implies
$$\cali_Nr(w_2v_1-w_1v_2)=|\bfw|(k_2v_1-k_1v_2).$$
This gives
\begin{equation}\label{reqn}
r=\frac{|\bfw|}{\cali_N}\frac{k_2v_1-k_1v_2}{w_2v_1-w_1v_2}
\end{equation}
which must be a nonnegative integer less than $\frac{|\bfw|}{\gcd(|\bfw|,\cali_N)}$. 
We can also solve Equations (\ref{phisoln}) for $\phi$ by eliminating $\frac{\cali_Nr}{|\bfw|}$ giving
\begin{equation}\label{phisoln2}
\phi =2\pi \frac{k_1w_2-k_2w_1}{w_2v_1-w_1v_2}.
\end{equation}
Next we write (\ref{reqn}) as
\begin{equation}\label{reqn3}
r=\Bigl(\frac{\frac{|\bfw|}{\gcd(|\bfw|,\cali_N)}}{\gcd(|w_2v_1-w_1v_2|,\frac{|\bfw|}{\gcd(|\bfw|,\cali_N)})}\Bigr)\Bigl(\frac{k_2v_1-k_1v_2}{\frac{\cali_N}{\gcd(|\bfw|,\cali_N)}\frac{w_2v_1-w_1v_2}{\gcd(|w_2v_1-w_1v_2|,\frac{|\bfw|}{\gcd(|\bfw|,\cali_N)})}}\Bigr)
\end{equation}
Since $v_1$ and $v_2$ are relatively prime, we can choose $k_1$ and $k_2$ so that the term in the last parentheses is $1$. This determines $r$ as 
\begin{equation}\label{reqn4}
r=\frac{\frac{|\bfw|}{\gcd(|\bfw|,\cali_N)}}{\gcd(|w_2v_1-w_1v_2|,\frac{|\bfw|}{\gcd(|\bfw|,\cali_N)})}
\end{equation}

Now suppose that $z_2=0$. Then generally we have $e^{iv_1\phi}=e^{-2\pi\frac{\cali_Nw_1}{|\bfw|}ri}$ for some $r=0,\ldots,\frac{|\bfw|}{\gcd(|\bfw|,\cali_N)}-1$ or equivalently $r=1,\ldots,\frac{|\bfw|}{\gcd(|\bfw|,\cali_N)}$. This gives 
\begin{equation}\label{z20}
\phi=2\pi(-\frac{\cali_Nw_1r}{v_1|\bfw|}+\frac{k}{v_1}).
\end{equation} 
A similar computation at $z_1=0$ gives 
\begin{equation}\label{z10}
\phi=2\pi(-\frac{\cali_Nw_2r'}{v_2|\bfw|}+\frac{k'}{v_2}).
\end{equation}
We are interested in when regularity can occur. For this we need the minimal angle at the two endpoints to be equal. This gives
$$-\frac{\cali_Nw_2r'}{v_2|\bfw|}+\frac{k'}{v_2}=-\frac{\cali_Nw_1r}{v_1|\bfw|}+\frac{k}{v_1}$$
for some choice of integers $k,k'$ and nonnegative integers $r,r'<\frac{|\bfw|}{\gcd(|\bfw|,\cali_N)}$. This gives
\begin{equation}\label{endpteqn}
\frac{-\cali_Nw_2r'+k'|\bfw|}{v_2}=\frac{-\cali_Nw_1r+k|\bfw|}{v_1}.
\end{equation}

\subsection{Periods of Reeb Orbits}
We assume that $\bfw\neq (1,1)$. We want to determine the periods of the orbits of the flow of the Reeb vector field defined by the weight vector $\bfv=(v_1,v_2)$. In particular, we want to know when there is a regular Reeb vector field in the $\bfw$-Sasaki cone.

Let us now generally determine the minimal angle, hence the generic period of the Reeb orbits, on the dense open subset $Z$ defined by $z_1z_2\neq 0$. For convenience we set
$$ s=\gcd(|w_2v_1-w_1v_2|,\frac{|\bfw|}{\gcd(|\bfw|,\cali_N)}), \qquad t=\frac{\cali_N}{\gcd(|\bfw|,\cali_N)}.$$

\begin{lemma}\label{generic3}
The minimal angle on $Z$ is $\frac{2\pi}{s}$.  Thus,  $S^1_\phi/\bbz_s$ acts freely on the dense open subset $Z$. 
\end{lemma}

\begin{proof}
We choose $k_1,k_2$ in Equation (\ref{reqn3}) so that the last parentheses equals $1$. This gives
$$t\frac{w_2v_1-w_1v_2}{s}=k_2v_1-k_1v_2.$$
Rearranging this becomes
$$(sk_2-tw_2)v_1=(sk_1-tw_1)v_2.$$
Since $v_1$ and $v_2$ are relatively prime this equation implies $sk_i=tw_i+mv_i$ for $i=1,2$ and some integer $m$. Putting this into Equation (\ref{phisoln2}) gives $\phi=\frac{2\pi m}{s}$, so the minimal angle is $\frac{2\pi}{s}$.
\end{proof}

We now investigate the endpoints defined by $z_2=0$ and $z_1=0$.

\begin{proposition}\label{regularprop}
The following hold:
\begin{enumerate}
\item The period on $Z$, namely $\frac{2\pi}{s}$, is an integral multiple of the periods at the endpoints. Hence, $S^1_\phi/\bbz_s$ acts effectively on $M_{l_1,l_2,\bfw}$.
\item The period at the endpoint $z_j=0$ is $2\pi\frac{\gcd(\cali_N, |\bfw|)}{v_i|\bfw|}$ where $i\equiv j+1\mod 2$. So the end points have equal periods if and only if $\bfv=(1,1)$. 
\item The $\bfw$-Sasaki cone contains a regular Reeb vector field if and only if  $\bfv=(1,1)$ and $\frac{w_1+w_2}{\gcd(\cali_N,w_1+w_2)}$ divides $w_1-w_2$.
\end{enumerate}
\end{proposition}

\begin{proof}
A Reeb vector field will be regular if and only if the period of its orbit is the same at all points. We know that it is $\frac{2\pi}{s}$ on $Z$. We need to determine the minimal angle at the endpoints. From Equation (\ref{z20}) the angle at $z_2=0$ is
$$\phi=2\pi(\frac{-\cali_Nw_1r+k|\bfw|}{v_1|\bfw|})=2\pi\gcd(\cali_N,|\bfw|)\Bigl(\frac{\frac{|\bfw|k}{\gcd(\cali_N,|\bfw|)}-\frac{\cali_Nw_1r}{\gcd(\cali_N,|\bfw|)}}{v_1|\bfw|}\Bigr). $$
Now 
$$\gcd(\frac{|\bfw|}{\gcd(\cali_N,|\bfw|)},\frac{\cali_Nw_1}{\gcd(\cali_N,|\bfw|)}=1,$$
so we can choose $k$ and $r$ such that numerator of the term in the large parentheses is $1$. This gives period $2\pi\frac{\gcd(\cali_N, |\bfw|)}{v_1|\bfw|}$. Similarly, at $z_1=0$ we have the period $2\pi\frac{\gcd(\cali_N, |\bfw|)}{v_2|\bfw|}$. So the period is the same at the endpoints if and only if $v_1=v_2$ which is equivalent to $\bfv=(1,1)$ since $v_1$ and $v_2$  are relatively prime which proves $(2)$. 

Moreover, the period is the same at all points if and only if
\begin{equation}\label{eqper}
\bfv=(1,1),\qquad \frac{|\bfw|}{\gcd(\cali_N,|\bfw|)}=s=\gcd(|w_2v_1-w_1v_2|,\frac{|\bfw|}{\gcd(|\bfw|,\cali_N)}).
\end{equation} 
But the last equation holds if and only if $\frac{|\bfw|}{\gcd(\cali_N,|\bfw|)}$ divides $w_1-w_2$ proving $(3)$. 

(1) follows from the fact that for each $i=1,2$, $\frac{v_i|\bfw|}{\gcd(\cali_N,|\bfw|)}$ is an integral multiple of $\gcd(|w_2v_1-w_1v_2|,\frac{|\bfw|}{\gcd(|\bfw|,\cali_N)})=s$.
\end{proof}

We have an immediate 
\begin{corollary}\label{cali1cor}
If $\cali_N=1$ there are no regular Reeb vector fields in any $\bfw$-Sasaki cone with $\bfw\neq (1,1)$.
\end{corollary}

Actually, one can say much more.
\begin{proposition}\label{numreg}
Assume $\bfw\neq (1,1)$. There are exactly $K-1$ different $\bfw$-Sasaki cones that have a regular Reeb vector field. These are given by
\begin{equation}\label{regw}
\bfw =\Bigl(\frac{K+n}{\gcd(K+n,K-n)},\frac{K-n}{\gcd(K+n,K-n)}\Bigr)
\end{equation}
where $K=\gcd(\cali_N,|\bfw|)$ and $1\leq n<K$.
\end{proposition}

\begin{proof}
By Proposition \ref{regularprop} a $\bfw$-Sasaki cone contains a regular Reeb vector field if and only if there is $n\in \bbz^+$ such that 
$$w_1-w_2=n\frac{w_1+w_2}{\gcd(\cali_N,w_1+w_2)}.$$
Clearly, for a solution we must have $n<\gcd(\cali_N,w_1+w_2)$. Then we have a solution if and only if 
$$(K-n)w_1=(K+n)w_2$$
for all $1\leq n<K$. Since $w_1>w_2$ and they are relatively prime we have the unique solution Equation \eqref{regw} for each integer $1\leq n<K$.
\end{proof}

\begin{example}\label{Findex}
Let us determine the $\bfw$-joins with regular Reeb vector field for $\cali_N=2,3$.
For example, if $\cali_N=2$ for a solution to Equation \eqref{regw} we must have $K=2$ which gives
$n=1$ and $\bfw=(3,1)$. This has as a consequence Corollary \ref{Ypqcor} below. 
Similarly if $\cali_N=3$ we must have $K=3$, which gives two solutions $\bfw=(2,1)$ and $\bfw=(5,1)$.
\end{example}

\begin{example}\label{Ypq}
Recall the contact structures $Y^{p,q}$ on $S^2\times S^3$ first studied in the context of Sasaki-Einstein metrics in \cite{GMSW04a}, where $p$ and $q$ are relatively prime positive integers satisfying $p>1$ and $q<p$.
Since in this case the manifold $M_{l_1,l_2,\bfw}=M^3\star_{l_1,l_2}S^3_\bfw$ is $S^2\times S^3$, we have $N=S^2$ with its standard (Fubini-Study) K\"ahler structure. Hence, $\cali_N=2$. 

This was treated in Example 4.7 of \cite{BoPa10} although the conventions\footnote{In particular, there we chose $w_1\leq w_2$; whereas, here we use the opposite convention, $w_1\geq w_2$.} are slightly different. Here we set
\begin{equation}\label{pqw}
\bfw=\frac{1}{\gcd(p+q,p-q)}\bigl(p+q,p-q\bigr).
\end{equation}
Note that the conditions on $p,q$ eliminate the case $\bfw=(1,1)$. We claim that the following relations hold:
\begin{equation}\label{pqrel}
l_1=\gcd(p+q,p-q),\qquad l_2=p.
\end{equation}
To see this we first notice that Lemma \ref{c10} implies that to have a Sasaki-Einstein metric we must have $l_1=1$ and $l_2=\frac{|\bfw|}{2}$ if $|\bfw|$ is even, and $l_1=2$ and $l_2=|\bfw|$ if $|\bfw|$ is odd. Thus, the second of Equations (\ref{pqrel}) follows from the first and Equation (\ref{pqw}). To prove the first equation we first notice that since $p$ and $q$ are relatively prime, $\gcd(p+q,p-q)$ is either $1$ or $2$. Next from Equation (\ref{pqw}) we note that $2p=\gcd(p+q,p-q)|\bfw|$. So if $|\bfw|$ is odd, then $\gcd(p+q,p-q)$ must be even, hence $2$. So the first equation holds in this case. This also shows that if $|\bfw|$ is odd then $p=|\bfw|$ is also odd. Now if $|\bfw|$ is even then both $p+q$ and $p-q$ must be odd. But  then $\gcd(p+q,p-q)$ must be $1$. In this case $p=\frac{|\bfw|}{2}$ which can be either odd or even.

Then from Proposition \ref{regularprop} we recover the following result of \cite{BoPa10}:
\begin{corollary}\label{Ypqcor}
For $Y^{p,q}$ the $\bfw$-Sasaki cone has a regular Reeb vector field if and only if $p=2,q=1$.
\end{corollary}
For the general $Y^{p,q}$ the Reeb vector field of Theorem 4.2 of \cite{BoPa10} is equivalent to choosing $\bfv=(1,1)$ here. However, as stated in Corollary \ref{Ypqcor}, it is regular only when $p=2,q=1$. Otherwise, it is quasi-regular with ramification index $m=m_1=m_2=p$ if $p$ is odd, and $m=\frac{p}{2}$ if $p$ is even.
We remark that the  quotient of $Y^{2,1}$ by the regular Reeb vector field is $\bbc\bbp^2$ blown-up at a point; whereas, we have arrived at it from the $\bfw$-Sasaki cone of an $S^1$ orbibundle over $S^2\times\bbc\bbp^1[3,1]$. 
\end{example}

\subsection{$B_{l_1,l_2,\bfv,\bfw}$ as a Log Pair}
We follow the analysis in Section 3 of \cite{BoTo13}. We have the action of the 2-torus $S^1_\phi/\bbz_s\times (S^1_\theta/\bbz_{l_2})$ on $M\times L(\frac{|\bfw|}{\gcd(|\bfw|,\cali_N)};l_1w_1,l_1w_2)$ given by Equation (\ref{T2action3}) whose quotient space is $B_{l_1,l_2,\bfv,\bfw}$. It follows from Equation (\ref{T2action3}) that $B_{l_1,l_2,\bfv,\bfw}$ is a bundle over $N$ with fiber a weighted projective space of complex dimension one. By (1) of Proposition \ref{regularprop} the generic period is an integral multiple, say $m_i$, of the period at the divisor $D_i$. Thus, for $i=1,2$ we have 
\begin{equation}\label{ramind}
m_i=v_i\frac{|\bfw|}{s\gcd(|\bfw|,\cali_N)}=v_im.
\end{equation}
Note that from its definition $m=\frac{l_2}{s}=\frac{|\bfw|}{s\gcd(|\bfw|,\cali_N)}$, so $m_i$ is indeed a positive integer. It is the ramification index of the branch divisor $D_i$. We think of $D_1$ as the zero section and $D_2$ as the infinity section of the bundle $B_{l_1,l_2,\bfv,\bfw}$. Thus, $B_{l_1,l_2,\bfv,\bfw}$ is a fiber bundle over $N$ with fiber $\bbc\bbp^1[v_1,v_2]/\bbz_m\approx \bbc\bbp^1$. The isomorphism is simply $[z_1,z_2]\mapsto [z_1^{m_2},z_2^{m_1}]$ where the brackets denote the obvious equivalence classes on $\bbc\bbp^1[v_1,v_2]/\bbz_m$. The complex structure of $B_{l_1,l_2,\bfv,\bfw}$ is the projection of the transverse complex structure on $M_{l_1,l_2,\bfw}$ which in turn is the lift of the product complex structure on $N\times \bbc\bbp^1[\bfw]$. However, $B_{l_1,l_2,\bfv,\bfw}$ is not generally a product as a complex orbifold, nor even topologically.

Now we can follow the analysis leading to Lemma 3.14 of \cite{BoTo13}. So we define the map 
$$\th_\bfv:M\times L(l_2;l_1w_1,l_1w_2)\ra{1.6} M\times L(l_2;l_1w_1v_2,l_1w_2v_1)$$ 
by
\begin{equation}\label{th}
\th_\bfv(x,u;[z_1,z_2])=(x,u;[z_1^{m_2},z_2^{m_1}]).
\end{equation}
It is a $mv_1v_2$-fold covering map. Similar to \cite{BoTo13} we get a commutative diagram:
\begin{equation}\label{actcomdia}
\begin{matrix}
M\times L(l_2;l_1w_1,l_1w_2) &\fract{\cala_{\bfv,l,\bfw}(\grl,\grt)}{\ra{2.5}} & M\times  L(l_2;l_1w_1,l_1w_2) \\
\decdnar{\th_\bfv} && \decdnar{\th_\bfv} \\
M\times  L(l_2;l_1w'_1,l_1w'_2) & \fract{\cala_{1,l,\bfw'}(\grl,\grt^{mv_1v_2})}{\ra{2.5}} & M\times  L(l_2;l_1w'_1,l_1w'_2),
\end{matrix}
\end{equation}
where $\bfw'=(v_2w_1,v_1w_2)$. So $B_{l_1,l_2,\bfv,\bfw}$ is the log pair $(S_n,\grD)$ with 
branch divisor 
\begin{equation}\label{branchdiv}
\grD=(1-\frac{1}{m_1})D_1+ (1-\frac{1}{m_2})D_2,
\end{equation}
where $S_n$ is a smooth $\bbc\bbp^1$-bundle over $N$, that is a ruled manifold as described in Section \ref{ruledsec}. Now $B_{l_1,l_2,\bfv,\bfw}$ is the quotient orbifold 
$$\bigl(M\times L(l_2;l_1w_1,l_1w_2)\bigr)/\cala_{\bfv,l,\bfw}(\grl,\grt),$$ 
and $B_{l_1,l_2,1,\bfw'}$ is the quotient $\bigl(M\times  L(l_2;l_1w'_1,l_1w'_2)\bigr)/\cala_{1,l,\bfw'}(\grl,\grt^{v_1v_2})$. So $\th_\bfv$ induces a map $h_\bfv:B_{l_1,l_2,\bfv,\bfw}\ra{1.6}S_n$ defined by 
\begin{equation}\label{hquot}
h_\bfv([x,u;[z_1,z_2]])=[x,u;[z_1^{m_2},z_2^{m_1}]],
\end{equation}
where the outer brackets denote the equivalence class with respect to the corresponding $T^2$ action. We have

\begin{lemma}\label{biholo}
The map $h_\bfv:B_{l_1,l_2,\bfv,\bfw}\ra{1.6}B_{l_1,l_2,1,\bfw'}$ defined by Equation (\ref{hquot}) is a biholomorphism.
\end{lemma}

\begin{proof}
The map is ostensibly holomorphic. Now $\th_\bfv$ is the identity map on $M$ and a $v_1v_2$-fold covering map on the corresponding lens spaces. From the commutative diagram (\ref{actcomdia}) the induced map $h_\bfv$ is fiber preserving and is a bijection on the fibers with holomorphic inverse.
\end{proof}

Lemma \ref{biholo} allows us to consider the orbifold $B_{l_1,l_2,\bfv,\bfw}$ as the log pair $(B_{l_1,l_2,1,\bfw'},\grD)$ where $\grD$ is given by Equation \eqref{branchdiv}. Notice that even when $\bfv=(1,1)$ the orbifold structure can be non-trivial, namely, $B_{l_1,l_2,(1,1),\bfw}=(S_n,\grD)$ where $m_1=m_2=m=\frac{|\bfw|}{s\gcd(|\bfw|,\cali_N)}$ and
$$\grD=(1-\frac{1}{m})(D_1+D_2).$$

The $T^2$ action $\cala_{1,\bfl,\bfw'}:M\times  L(l_2;l_1w'_1,l_1w'_2)\ra{1.5} M\times  L(l_2;l_1w'_1,l_1w'_2)$ is given by
\begin{equation}\label{T2action4}
(x,u;z_1,z_2)\mapsto (x,e^{i\theta}u;[e^{i(\phi-\frac{l_1w'_1}{l_2}\theta)}z_1,e^{i(\phi-\frac{l_1w'_2}{l_2}\theta)}z_2]),
\end{equation}
Defining $\chi=\phi-\frac{l_1w'_1}{l_2}\theta$ gives
\begin{equation}\label{T2action5}
(x,u;z_1,z_2)\mapsto (x,e^{i\theta}u;[e^{i\chi}z_1,e^{i(\chi+\frac{l_1}{l_2}(w'_1-w'_2)\theta)}z_2]).
\end{equation}
The analysis above shows that this action is generally not free, but has branch divisors at the zero ($z_2=0$) and infinity ($z_1=0$) sections with ramification indices both equal to $m$. 

Equation (\ref{T2action5}) tells us that the $T^2$-quotient space $B_{l_1,l_2,1,\bfw'}$ is the projectivization of the holomorphic rank two vector bundle $E=\BOne \oplus L_n$ over $N$ where $\BOne$ denotes the trivial line bundle and $L_n$ is a line bundle of `degree' $n=\frac{l_1}{s}(w_1v_2-w_2v_1)$ with $s=\gcd(|w_1v_2-w_2v_1|,l_2)$. So $S_n=\bbp(\BOne\oplus L_n)$ is a smooth projective algebraic variety. Next we identify $N$ with the zero section $D_1$ of $L_n$, and note that $c_1(L_n)$ is just the restriction of the Poincar\'e dual  of $D_1$ to $D_1$, i.e. $PD(D_1)|_{D_1}=c_1(L_n)$.

Summarizing we have

\begin{theorem}\label{preSE}
Let $M_{l_1,l_2,\bfw}=M\star_{l_1,l_2}S^3_\bfw$ be the join as described in the beginning of the section with the induced contact structure $\cald_{l_1,l_2,\bfw}$ satisfying $c_1(\cald_{l_1,l_2,\bfw})=0$. Let $\bfv=(v_1,v_2)$ be a weight vector with relatively prime integer components and let $\xi_\bfv$ be the corresponding Reeb vector field in the Sasaki cone $\gt^+_2$. Then the quotient of $M_{l_1,l_2,\bfw}$ by the flow of the Reeb vector field $\xi_\bfv$ is a projective algebraic orbifold written as a the log pair $(S_n,\grD)$ where $S_n$ is the total space of the projective bundle $\bbp(\BOne\oplus L_n)$ over the Fano manifold $N$ with $n=\frac{l_1}{s}(w_1v_2-w_2v_1)$, $\grD$ the branch divisor
\begin{equation}\label{branchdiv2}
\grD=(1-\frac{1}{m_1})D_1+ (1-\frac{1}{m_2})D_2,
\end{equation}
with ramification indices $m_i=v_i\frac{|\bfw|}{s\gcd(|\bfw|,\cali_N)}=v_im$ and divisors $D_1$ and $D_2$ given by the zero section $\BOne\oplus 0$ and infinity section $0\oplus L_n$, respectively. Here $\cali_N$ is the Fano index of $N$, $s=\gcd(|w_1v_2-w_2v_1|,\frac{|\bfw|}{\gcd(\cali_N,|\bfw|})$, and $l_i$ are the relative indices given by $l_1=\frac{\cali_N}{\gcd(\cali_N,|\bfw|)}$ and $l_2=\frac{|\bfw|}{\gcd(\cali_N,|\bfw|)}$.

The fiber of the orbifold $(S_n,\grD)$ is the orbifold $\bbc\bbp^1[v_1,v_2]/\bbz_m$.
\end{theorem}

Notice that in this case we have $\pi_1^{orb}(S_n,\grD)=\bbz_m$.

\begin{example}
Consider $\bfw=(11,1)$ and $\bfv=(4,5)$. Let us take $\cali_N=1$. So we have $l_1=1,l_2=12,$ and $s=\gcd(12,55-4)=3$. Thus, the generic period of the Reeb vector field $\xi_\bfv$ is $\frac{2\pi}{3}$. The period at the endpoint $z_2=0$ (on the branch divisor $D_1$) is $\frac{2\pi}{48}$, and at the endpoint $z_1=0$ (on $D_2)$ it is $\frac{2\pi}{60}$. So there is an effective action of $S^1_\phi/\bbz_3$. This gives an isotropy of $\bbz_{16}$ on the branch divisor $D_1$, and an isotropy of $\bbz_{20}$ on $D_2$. The ramification indices are then $m_1=16$, $m_2=20$, and $m=\frac{12}{3}=4$. The projective bundle $S_{n}=\bbp(\BOne\oplus L_n)$ has $n=\frac{51}{3}=17$. So our log pair is $(S_{17},\grD)$ with 
$$\grD=\frac{15}{16}D_1+\frac{19}{20}D_2.$$
The fiber $F$ is a quotient of the orbifold weighted projective space $\bbc\bbp^1[4,5]$, namely $F=\bbc\bbp^1[4,5]/\bbz_4$. Here $\pi_1^{orb}(F)=\bbz_4$.

Now let us consider $\bfw=(11,1)$ and $\bfv=(1,1)$. Again take $\cali_N=1$, so $l_1=1, l_2=12$, and $s=\gcd(12,11-1)=2.$ So the  generic period of $\xi_\bfv$ is $\frac{2\pi}{2}=\pi$. The period on the branch divisors $D_1$ and $D_2$ are both $\frac{2\pi}{12}=\frac{\pi}{6}$, and the ramification index $m=6$. The projective bundle $S_n=\bbp(\BOne\oplus L_n)$ has $n=5$, so our log pair is $(S_5,\grD)$ with branch divisor
$$\grD= \frac{5}{6}(D_1+D_2).$$
In this case the fiber is a global quotient (developable) orbifold, namely $F=\bbc\bbp^1/\bbz_6$ with $\pi_1^{orb}(F)=\bbz_6$.
\end{example}

\begin{example}\label{Ypq2} 
This is a continuation of Example \ref{Ypq}. We take $\bfv=(1,1)$. Then 
$$s=\gcd(l_2,w_1-w_2)=\gcd(p,\frac{2q}{l_1}).$$
So $s=1$ if $l_1=2$ that is when $|\bfw|$ is odd which also implies that $p$ is odd. Whereas, if $|\bfw|$ is even, $l_1=1$, so $s=\gcd(p,2q)$ which is $2$ if $p$ is even and $1$ if $p$ is odd. Now consider $n$. We have 
$$n=\frac{l_1}{s}(w_1-w_2)=\frac{2q}{s}.$$
Thus, $n=2q$ when $p$ is odd, and $n=q$ when $p$ is even in which case $q$ must be odd. Moreover, from Equation (\ref{ramind}) we have
$$m=\frac{|\bfw|}{s\gcd(|\bfw|,\cali_N)}= 
\begin{cases} p~ \text{if $p$ is odd} \\
                \frac{p}{2} ~\text{if $p$ is even.}
                \end{cases}$$
So in this case our log pair is $(S_{2q},\grD)$ with 
$$\grD=\bigl(1-\frac{1}{p}\bigr)(D_1+D_2)$$
if $p$ is odd, and $(S_q,\grD)$ with $q$ odd and
$$\grD=\bigl(1-\frac{2}{p}\bigr)(D_1+D_2)$$
if $p$ is even. Here $S_{2q}$ and $S_q$ are the even and odd Hirzebruch surfaces, respectively. Note for $Y^{2,1}$ there is no branch divisor, so it is regular as we know from Corollary \ref{Ypqcor} and $S_1$ is $\bbc\bbp^2$ blown-up at a point.
\end{example}

\subsection{Examples with $N$ a del Pezzo Surface}
The Fano manifolds of complex dimension $2$ are usually called {\it del Pezzo surfaces}. They are exactly $\bbc\bbp^2,\bbc\bbp^1\times \bbc\bbp^1$, and $\bbc\bbp^2$ blown-up at $k$ generic points with $1\leq k\leq 8$. The join will be a Sasaki 7-manifold for these cases.

\begin{example}\label{Ncp2}
For $N=\bbc\bbp^2$ with its standard Fubini-Study K\"ahlerian structure, we have $\cali_N=3$. From Example \ref{Findex} we see that we have a regular Reeb vector field in the $\bfw$-Sasaki cone in precisely two cases, either $\bfw=(2,1)$, or $\bfw=(5,1)$. In the first case the relative Fano indices are $(l_1,l_2)=(1,1)$ while in the second case they are $(l_1,l_2)=(1,2)$. In the former case our 7-manifold $M^7_{(2,1)}=S^5\star_{1,1}S^3_{(2,1)}$ is an $S^3$-bundle over $\bbc\bbp^2$; whereas, in the latter case the 7-manifold $M^7_{(5,1)}=S^5\star_{1,2}S^3_{(5,1)}$ is an $L(2;5,1)$ bundle over $\bbc\bbp^2$. Moreover, it follows from standard lens space theory that $L(2;5,1)$ is diffeomorphic to the real projective space $\bbr\bbp^3$.

\end{example}

\begin{example}\label{2cp1}
For $N=\bbc\bbp^1\times \bbc\bbp^1$ with its standard Fubini-Study K\"ahlerian structure, we have $\cali_N=2$. There is only one case with a regular Reeb vector field, and that is $\bfw=(3,1)$. Here the relative Fano indices are $(1,2)$. In this case the 7-manifold is $(S^2\times S^3)\star_{1,2}S^3_{(3,1)}$ which can be realized as an $L(2;3,1)\approx \bbr\bbp^3$ lens space bundle over $\bbc\bbp^1\times\bbc\bbp^1$.
\end{example}

\begin{example}\label{blowups}
We take $N$ to be $\bbc\bbp^2$ blown-up at $k$ generic points where $k=1,\ldots,8$, or equivalently $N=N_k=\bbc\bbp^2\#k\overline{\bbc\bbp}^2$. All the K\"ahler structures have an extremal representative, but for $k=1,2$ they are not CSC. However, for $k=3,\ldots,8$ they are CSC, and hence, K\"ahler-Einstein. Notice that when $4\leq k\leq 8$ the complex automorphism group has dimension $0$, so the $\bfw$-Sasaki cone is the entire Sasaki cone. Moreover, if $5\leq k\leq 8$ the local moduli space has positive dimension, and we can choose any of the complex structures. By a theorem of Kobayashi and Ochiai \cite{KoOc73} we have $\cali_{N_k}=1$ for all $k=1,\ldots,8$. So $l_1=1,l_2=|\bfw|$, and by Corollary \ref{cali1cor} there are no regular Reeb vector fields in the $\bfw$-Sasaki cone with $\bfw\neq (1,1)$. In particular, if $4\leq k\leq 8$, there are no regular Reeb vector fields in the Sasaki cone. Generally, these are $L(|\bfw|;w_1,w_2)$ lens space bundles over $N_k$. Of course, the case $\bfw=(1,1)$ is just an $S^1$-bundle over $N_k\times \bbc\bbp^1$ with the product complex structure which is automatically regular. These were studied in \cite{BG00a}.
\end{example}

\section{The Topology of the Joins}

Since we are interested in compact Sasaki-Einstein manifolds, which have finite fundamental group, we shall assume that the Sasaki manifold $M$ is simply connected. It is then easy to construct examples with cyclic fundamental group. From the homotopy exact sequence of the fibration $S^1\ra{1.5}M\times S^3\ra{1.5} M_{l_1,l_2,\bfw}$ we have
\begin{proposition}\label{simcon}
If $M$ is simply connected, then so is $M_{l_1,l_2,\bfw}$. Moreover, if $M$ is 2-connected, $\pi_2(M_{l_1,l_2,\bfw})\approx \bbz$.
\end{proposition}

We now describe our method for computing the cohomology ring of the join $M_{l_1,l_2,\bfw}$.

\subsection{The Method}
Our approach uses the spectral sequence method employed in \cite{WaZi90,BG00a} (see also Section 7.6.2 of \cite{BG05}). The fibration $\pi_L$ in Diagram (\ref{s2comdia}) together with the torus bundle with total space $M\times S^3_\bfw$ gives the commutative diagram of fibrations
\begin{equation}\label{orbifibrationexactseq}
\begin{matrix}M\times S^3_\bfw &\ra{2.6} &M_{l_1,l_2,\bfw}&\ra{2.6}
&\mathsf{B}S^1 \\
\decdnar{=}&&\decdnar{}&&\decdnar{\psi}\\
M\times S^3_\bfw&\ra{2.6} & N\times\mathsf{B}\bbc\bbp^1[\bfw]&\ra{2.6}
&\mathsf{B}S^1\times \mathsf{B}S^1\, 
\end{matrix} \qquad \qquad
\end{equation}
where $\mathsf{B}G$ is the classifying space of a group $G$ or Haefliger's classifying space \cite{Hae84} of an orbifold if $G$ is an orbifold. Note that the lower fibration is a product of fibrations. In particular,  the fibration 
\begin{equation}\label{cporbfib}
S^3_\bfw \ra{2.6} \mathsf{B}\bbc\bbp^1[\bfw]\ra{2.6} \mathsf{B}S^1
\end{equation}
is rationally equivalent to the Hopf fibration, so over $\bbq$ the only non-vanishing differentials in its Leray-Serre spectral sequence are $d_4(\grb)=s^2$ where $\grb$ is the orientation class of $S^3$ and $s$ is a basis in $H^2( \mathsf{B}S^1,\bbq)\approx \bbq$ and those induced from $d_4$ by naturality. However, we want the cohomology over $\bbz$. 

\begin{lemma}\label{cporbcoh}
For $w_1$ and $w_2$ relatively prime positive integers we have
$$H^r_{orb}(\bbc\bbp^1[\bfw],\bbz)=H^r( \mathsf{B}\bbc\bbp^1[\bfw],\bbz)= \begin{cases}
                    \bbz &\text{for $r=0,2$,}\\                  
                    \bbz_{w_1w_2} &\text{for $r>2$ even,}\\
                     0 &\text{for $r$ odd.}
                     \end{cases}$$           
\end{lemma}

\begin{proof}
As in \cite{BG05} we cover the $\mathsf{B}\bbc\bbp^1[\bfw]$ with two overlapping open sets $p^{-1}(U_i)\approx \tU_i\times_{\grG_i}EO$ where $U_i$ is $\bbc\bbp^1\setminus \{0\}$ and $\bbc\bbp^1\setminus \{\infty\}$ for $i=1,2$, respectively. The Mayer-Vietoris sequence is
$$\ra{.8}H^r(\mathsf{B}\bbc\bbp^1[\bfw],\bbz)\ra{1.0} H^r(p^{-1}(U_1),\bbz)\oplus H^r(p^{-1}(U_2)\ra{1.0} H^r(p^{-1}(U_1)\cap p^{-1}(U_2),\bbz)\ra{.4}\cdots$$
Now $p^{-1}(U_i)\approx \tU_i\times_{\grG_i}EO$ is the Eilenberg-MacLane space $K(\bbz_{w_i},1)$ whose cohomology is the group cohomology
$$H^r(\bbz_{w_i},\bbz)=\begin{cases}
                    \bbz &\text{for $r=0$,}\\
                    \bbz_{w_i} &\text{for $r>0$ even,}\\
                     0 &\text{for $r$ odd.}
\end{cases}$$
Moreover, $p^{-1}(U_1)\cap p^{-1}(U_2)=\tU_1\cap\tU_2\times_{\grG_1\cap\grG_2}EO$ and  since $w_1,w_2$ are relatively prime $\grG_1\cap\grG_2=\bbz_{w_1}\cap\bbz_{w_2}=\{\BOne\}$.
So for $r=2$ the Mayer-Vietoris sequence becomes
\begin{equation}\label{LS2}
0\ra{1.8}\bbz\fract{j}{\ra{1.8}}H^2(\mathsf{B}\bbc\bbp^1[\bfw],\bbz)\ra{1.8}\bbz_{w_1w_2}\ra{1.8} 0.
\end{equation}
From the $E_2$ term of the Leray-Serre spectral sequence of the fibration (\ref{cporbfib}), we see that the map $j$ in (\ref{LS2}) must be multiplication by $w_1w_2$ implying that $H^2(\mathsf{B}\bbc\bbp^1[\bfw],\bbz)\approx \bbz$.

For $r>2$ even the sequence gives $H^r(\mathsf{B}\bbc\bbp^1[\bfw],\bbz)\approx \bbz_{w_1}\oplus \bbz_{w_2}\approx \bbz_{w_1w_2}$, whereas for $r$ odd $H^r(\mathsf{B}\bbc\bbp^1[\bfw],\bbz)\approx 0$.
\end{proof} 

One now easily sees that 

\begin{lemma}\label{LS}
The only non-vanishing differentials in the Leray-Serre spectral sequence of the fibration (\ref{cporbfib}) are those induced naturally by $d_4(\gra)= w_1w_2s^2$ for $s\in H^2(\mathsf{B}S^1,\bbz)\approx \bbz[s]$ and $\gra$ the orientation class of $S^3$.       
\end{lemma}

Now the map $\psi$ of Diagram (\ref{orbifibrationexactseq}) is that induced by the inclusion $e^{i\theta}\mapsto (e^{il_2\theta},e^{-il_1\theta})$. So noting 
$$H^*(\mathsf{B}S^1\times \mathsf{B}S^1,\bbz)=\bbz[s_1,s_2]$$ 
we see that $\psi^*s_1=l_2s$ and $\psi^*s_2=-l_1s$. This together with Lemma \ref{LS} gives $d_4(\gra)=w_1w_2l_1^2s^2$ in the Leray-Serre spectral sequence of the top fibration in Diagram (\ref{orbifibrationexactseq}).

Further analysis depends on the differentials in the spectral sequence of the fibration 
\begin{equation}\label{MNspec}
M\ra{1.5}N\ra{1.5}\mathsf{B}S^1. 
\end{equation}

\begin{algorithm}
Given the differentials in the spectral sequence of the fibration (\ref{MNspec}), one can use the commutative diagram (\ref{orbifibrationexactseq}) to compute the cohomology ring of the join manifold $M_{l_1,l_2,\bfw}$.
\end{algorithm}

\subsection{Examples in General Dimension}
One case that is particularly easy to describe in all odd dimensions is when $M$ is the odd-dimensional sphere $S^{2r+1}$ with $r=2,3,\ldots,$. Here we have $N=\bbc\bbp^r$ in which case we have $\cali_N=r+1$. So the relative Fano indices are 
\begin{equation}\label{relFanind}
(l_1(\bfw),l_2(\bfw))=\bigl(\frac{r+1}{\gcd(|\bfw|,r+1)},\frac{|\bfw|}{\gcd(|\bfw|,r+1)}\bigr).
\end{equation}

Since $l_1,l_2$ are uniquely determined by $r$ and $\bfw$, we write our joins as $M^{2r+3}_\bfw$. We have\begin{theorem}\label{topcpr}
The join $M^{2r+3}_{\bfw}=S^{2r+1}\star_{l_1,l_2}S^3_\bfw$ with relative Fano indices given by Equation (\ref{relFanind}) has integral cohomology ring
given by
$$\bbz[x,y]/(w_1w_2l_1(\bfw)^2x^2,x^{r+1},x^2y,y^2)$$
where $x,y$ are classes of degree $2$ and $2r+1$, respectively.
\end{theorem}

\begin{proof}
The $E_2$ term of the Leray-Serre spectral sequence of the top fibration of diagram (\ref{orbifibrationexactseq}) is 
$$E^{p,q}_2=H^p(\mathsf{B}S^1,H^q(S^{2r+1}\times S^3_\bfw,\bbz))\approx \bbz[s]\otimes\grL[\gra,\grb],$$ 
where $\gra$ is a $3$-class and $\grb$ is a $2r+1$ class. By the Leray-Serre Theorem this converges to $H^{p+q}(M_{\bfw}^{2r+3},\bbz)$.
From the usual Hopf fibration and Lemma \ref{LS} the only non-zero differentials in the Leray-Serre spectral sequence of the bottom fibration in Diagram (\ref{orbifibrationexactseq}) are $d_4(\gra)=w_1w_2s^2_2$ and $d_{2r+2}(\grb)=s^{r+1}_1$. By naturality the differentials of the top fibration of (\ref{orbifibrationexactseq}) are $d_4(\gra)=w_1w_2(-l_1s)^2$ and $d_{2r+2}(\grb)=(l_2s)^{r+1}$. It follows that $H^{p}(M_{\bfw}^{2r+3},\bbz)$ has an element $x$ of degree $2$ with $w_1w_2l_1^2x^2$ vanishing, and since $l_2$ is relatively prime to $w_1w_2l_1^2$,  $x^p$ vanishes for $p\geq r+1$. Similarly, for dimensional reasons there is an element $y$ of degree $2r+1$ such that $y^2$ and $x^2y$ vanish.
\end{proof}

The connected component $\gA\gu\gt_0(M^{2r+3}_{\bfw})$ of the Sasaki automorphism group $\gA\gu\gt(M^{2r+3}_{\bfw})$ is $SU(r+1)\times T^2$, so these are toric Sasaki manifolds. However, our methods only make essential use of the 2-dimensional $\bfw$-subtorus. We shall make use of the following
\begin{lemma}\label{H4lem}
If $H^4(M^{2r+3}_\bfw,\bbz)=H^4(M^{2r+3}_{\bfw'},\bbz)$, then $w'_1w'_2=w_1w_2$ and $l_1(\bfw')=l_1(\bfw)$.
\end{lemma}

\begin{proof}
The equality of the 4th cohomology groups together with the definition of $l_1$ imply
$$w'_1w'_2l_1\gcd(|\bfw|,r+1)^2=w_1w_2\gcd(|\bfw'|,r+1)^2.$$
Set $g_\bfw=\gcd(|\bfw|,r+1)$ and $g_{\bfw'}=\gcd(|\bfw'|,r+1)$. Assume $g_{\bfw'}>1$. Since $\gcd(w'_1,w'_2)=1$, $g_{\bfw'}$ does not divide $w'_1w'_2$. Thus, $g_{\bfw'}^2$ divides $g_\bfw^2$. Interchanging the roles of $\bfw'$ and $\bfw$ gives $g_{\bfw'}=g_\bfw$ which implies $l_1(\bfw')=l_1(\bfw)$, and hence, the lemma in the case that $g_{\bfw'}>1$. Now assume $g_{\bfw'}=1$. Then we have $w_1w_2=w'_1w'_2g_\bfw^2$ which implies that $g_\bfw$ divides $w_1w_2$. But then since $w_1,w_2$ are relatively prime, we must have $g_\bfw=1$.
\end{proof}

Let us set $W=w_1w_2$, and write the prime decomposition of $W=w_1w_2=p_1^{a_1}\cdots p_k^{a_k}$ Let $P_k$ be the number of partitions  of $W$ into the product $w_1w_2$ of unordered relatively prime integers, including the pair $(w_1w_2,1)$. Then a counting argument gives $P_k=2^{k-1}$. Once counted we then order the pair $w_1>w_2$ as before. Let $\calp_W$ denote the set of $(2r+3)$-manifolds $M^{2r+3}_\bfw$ with isomorphic cohomology rings. Then Lemma \ref{H4lem} implies that the cardenality of $\calp_W$ is $P_k=2^{k-1}$. This proves Corollary \ref{cor1} on the Introduction.

Generally we can construct the join of any Sasaki-Einstein manifold with the standard $S^3$ to obtain new Sasaki-Einstein manifolds as done in \cite{BG00a}. In the present paper we take the join with a weighted $S^3_\bfw$ and then deform in the Sasaki-cone. From the topological viewpoint this usually adds torsion coming from the effect of Lemmas \ref{cporbcoh} and \ref{LS}. However, in the simplest example which occurs in dimension $5$ the differentials in the spectral sequence conspire to cancel the occurrence of torsion. Of course, in this dimension the only positive K\"ahler-Einstein $2$-manifold is $N=\bbc\bbp^1$ with its standard Fubini-Study K\"ahler-Einstein structure. Then our procedure gives the $5$-manifolds $Y^{p,q}$ discovered by the physicists \cite{GMSW04a} which are diffeomorphic to $S^2\times S^3$ for all relatively prime positive integers $p,q$ such that $1<q<p$. This case has been well studied from various perspectives \cite{Boy11,BoPa10,MaSp05b,CLPP05}. The $7$-dimensional case is quite amenable to further study, and we shall concentrate our efforts in this direction.

It is worth mentioning that the finiteness of deformation types of smooth Fano manifolds implies a bound on the Betti numbers of the join which only depends on dimension. This gives a Betti number bound on the manifolds obtained from our construction. In particular, in dimension seven $b_2(M^7_{k,\bfw})\leq 9$ as seen explicitly above, whereas, in dimension nine we have the bound $b_2(M^9_{k,\bfw})\leq 10$ \cite{BG00a}.

\subsection{Examples in Dimension $7$}
Here we consider circle bundles over the del Pezzo surfaces $\bbc\bbp^2,\bbc\bbp^1\times \bbc\bbp^1$ and $\bbc\bbp^2\# k\overline{\bbc\bbp}^2$ for $k=1,\ldots, 8$. As shown in the next section in each case the Sasaki cone admits an Sasaki-Einstein metric for all admissible $\bfw$.

\subsubsection{$M=S^5$}
In this case $N=\bbc\bbp^2$ and $S^5\ra{1.6} \bbc\bbp^2$ is the standard Hopf fibration and $\cali_N=3$. This is a special case of Proposition \ref{topcpr}. 

\begin{proposition}\label{topcp2}
Let $M^7_{\bfw}$ be a simply connected 7-manifold of Theorem \ref{topcpr}. There are two cases:
\begin{enumerate}
\item $3$ divides $|\bfw|$ which implies $l_2=\frac{|\bfw|}{3}$ and $l_1=1$.
\item $3$ does not divide $|\bfw|$ in which case $l_2=|\bfw|$ and $l_1=3$.
\end{enumerate}
In both cases the cohomology ring is given by
$$\bbz[x,y]/(w_1w_2l_1^2x^2,x^3,x^2y,y^2)$$
where $x,y$ are classes of degree $2$ and $5$, respectively.
\end{proposition}

Notice that in case (1) of Proposition \ref{topcp2} $H^4(M^7_\bfw,\bbz)=\bbz_{w_1w_2}$, whereas in case (2) we have $H^4(M^7_{\bfw},\bbz)=\bbz_{9w_1w_2}$. Since $3$ must divide $w_1+w_2$ in the first case and $w_1w_2$ are relatively prime, their cohomology rings are never isomorphic.

\begin{remark}
Let us make a brief remark about the homogeneous case $\bfw=(1,1)$ with symmetry group $SU(3)\times SU(2)\times U(1)$. There is a unique solution with a Sasaki-Einstein metric as shown in \cite{BG00a}. However, dropping both the Einstein and Sasakian conditions, Kreck and Stolz \cite{KS88} gave a diffeomorphism and homeomorphism classification. Furthermore, using the results of \cite{WaZi90}, they show that in certain cases each of the 28 diffeomorphism types admits an Einstein metric. If we drop the Einstein condition and allow contact bundles with non-trivial $c_1$ we can apply the classification results of \cite{KS88} to the Sasakian case. This will be studied elsewhere. 
\end{remark}

For dimension 7 we see from Proposition \ref{topcp2} that if $3$ divides $w_1+w_2$ then the order $|H^4|$ is $W$. However, if $3$ does not divide $w_1+w_2$ then the order of $|H^4|$ is $9W$. So by Lemma \ref{H4lem} $\calp_W$ splits into two cases, $\calp_W^0$ if $W+1$ is divisible by $3$, and $\calp_W^1$ if $W+1$ is not divisible by $3$. Of course, in either case the cardenality of $\calp_W$ is $2^{k-1}$ where $k$ is the number of prime powers in the prime decomposition of $W$.

\begin{proposition}\label{homequivprop}
Suppose the order of $H^4$ is odd. The elements $M^7_\bfw$ and $M^7_{\bfw'}$ in $\calp_W^0$ are homotopy inequivalent if and only if either 
\begin{equation*}
\bigl(\frac{w'_1+w'_2}{3}\bigr)^3\equiv \pm\bigl(\frac{w_1+w_2}{3}\bigr)^3 \mod \bbz_{W}.
\end{equation*}
The elements $M^7_\bfw$ and $M^7_{\bfw'}$ in $\calp_W^1$ are homotopy inequivalent if and only if  
\begin{equation*}
(w'_1+w'_2)^3 \equiv\pm (w_1+w_2)^3 \mod \bbz_{9W}.
\end{equation*}
\end{proposition}      

\begin{proof}
For $r=2$ consider the $E_6$ differential $d_6(\grb)=l_2(\bfw)^3s^3$ in the spectral sequence of Proposition \ref{topcpr}. Since $l_2$ is relatively prime to $l_1(\bfw)^2w_1w_2$, this takes values in the multiplicative group $\bbz_{l_1^2W}^*$ of units in $\bbz_{l_1^2W}$. Taking into account the choice of generators, it takes its values in $\bbz_{l_1^2W}^*/\{\pm 1\}$. According to Theorem 5.1 of \cite{Kru97} $M^7_\bfw,M^7_{\bfw'}\in\calp_W$ are homotopy equivalent if and only if $l_2(\bfw')^3= l_2(\bfw)^3$ in $\bbz^*_{l_1^2W}/\{\pm 1\}$. Of course, this means that $l_2(\bfw')^3=\pm l_2(\bfw)^3$ in $\bbz^*_{l_1^2W}$. Note that the the other two conditions of Theorem 5.1 of \cite{Kru97} are automatically satisfied in our case.
\end{proof}

Using a Maple program we have checked some examples for homotopy equivalence which appears to be quite sparse. So far we haven't found any examples of a homotopy equivalence. However, we have not done a systematic computer search which we leave for future work.

\begin{example}\label{ex1} Our first example is an infinite sequence of pairs with the same cohomology ring. Set $W=3p$ with $p$ an odd prime not equal to $3$, which gives $P_k=2$. Then for each odd prime $p\neq 3$ there are  two manifolds in $\calp_W^1$, namely $M^7_{(3p,1)}$ and $M^7_{(p,3)}$. The order of $H^4$ is $27p$. We check the conditions of Proposition \ref{homequivprop}. We find
$$(3p+1)^3\equiv 9p+1 \mod 27p, \qquad (p+3)^3\equiv p^3+9p^2+27 \mod 27p.$$
First we look for integer solutions of $p^3+9p^2-9p+26\equiv 0 \mod 27p.$ By the rational root test the solutions could only be $p=2,13,26$ none of which are solutions. Next we check the second condition of Proposition \ref{homequivprop}, namely, 
$p^3+9p^2+9p+28\equiv 0 \mod 27p.$
Again by the rational root test we find the only possibilities are $p=2,7,14,28$, from which s we see that there are no solutions. Thus, we see that $M^7_{(3p,1)}$ and $M^7_{(p,3)}$ are not homotopy equivalent for any odd $p\neq 3$.

By the same arguments one can also show that the infinite sequence of pairs of the form $M^7_{(9p,1)}$ and $M^7_{(p,9)}$, with $p$ an odd prime relatively prime to $3$, are never homotopy equivalent.
\end{example}

\begin{remark}\label{pairrem}
In Example \ref{ex1} we do not need to have $p$ a prime, but we do need it to be relatively prime to $3$. In this more general case, there will be more elements in $\calp_W^1$. For example, if $p=55$ we have $P_k=4$ and the pair $(M^7_{(165,1)},M^7_{(55,3)}$ has the same cohomology ring as $M^7_{(33,5)}$ and $M^7_{(15,11)}$. However, they are not homotopy equivalent to either member of the pair nor to each other.
\end{remark}

\begin{example}\label{ex2} A somewhat more involved example is obtained by setting $W=5\cdot 7\cdot 11\cdot 17$. Here $P_k=8$, so this gives eight 7-manifolds in $\calp_W^0$, namely, 
$$M^7_{(6545,1)},M^7_{(1309,5)},M^7_{(935,7)},M^7_{(595,11)},M^7_{(385,17)},M^7_{(187,35)},M^7_{(119,55)},,M^7_{(85,77)}. $$
One can check that these do not satisfy the conditions for homotopy equivalence of Proposition \ref{homequivprop}. So they are all homotopy inequivalent.
\end{example}

It is easy to get a necessary condition for homeomorphism.

\begin{proposition}
Suppose $w_1'w_2'=w_1w_2$ is odd and that $M^7_\bfw$ and $M^7_{\bfw'}$ are homeomorphic. Then in addition to the conditions of Proposition \ref{homequivprop}, we must have 
$$2(w'_1+w'_2)^2\equiv 2(w_1+w_2)^2 \mod 3w_1w_2.$$
\end{proposition}

\begin{proof}
This is because the first Pontrjagin class $p_1$ is actually a homeomorphism invariant\footnote{This appears to be a folklore result with no proof anywhere in the literature. It is stated without proof on page 2828 of \cite{Kru97} and on page 31 of \cite{KrLu05}. We thank Matthias Kreck for providing us with a proof that $p_1$ is a homeomorphism invariant.}. From Kruggel \cite{Kru97} we see that if $3$ does not divide $|\bfw|$
\begin{equation}\label{p1}
p_1(M^7_\bfw)\equiv 3|\bfw|^2-9w_1^2-9w_2^2\equiv -6|\bfw|^2 \mod 9w_1w_2,
\end{equation}
which implies the result in this case. If $3$ divides $|\bfw|$ we have 
\begin{equation}\label{p1'}
p_1(M^7_\bfw)\equiv -6\Bigl(\frac{|\bfw|}{3}\Bigr)^2 \mod w_1w_2
\end{equation}
and this implies the same result.
\end{proof}

Note that Equations \eqref{p1} and \eqref{p1'} both imply the third condition of Theorem 5.1 in \cite{Kru97} holds in our case. To determine a full homeomorphism and diffeomorphism classification requires the Kreck-Stolz invariants \cite{KS88} $s_1,s_2,s_3\in \bbq/\bbz$. These can be determined as functions of $\bfw$ in our case by using the formulae in \cite{Esc05,Kru05}; however, they are quite complicated and the classification requires computer programing which we leave for future work.

It is interesting to compare the Sasaki-Einstein 7-manifolds described by Theorem \ref{setop} with the 3-Sasakian 7-manifolds studied in \cite{BGM94,BG99} for their cohomology rings have the same form. Seven dimensional manifolds whose cohomology rings are of this type were called 7-manifolds of type $r$ in \cite{Kru97} where $r$ is the order of $H^4$.  First recall that the 3-Sasakian 7-manifolds in \cite{BGM94} are given by a triple of pairwise relatively prime positive integers $(p_1,p_2,p_3)$ and $H^4$ is isomorphic to $\bbz_{\grs_2(\bfp)}$ where $\grs_2(\bfp)=p_1p_2+p_1p_3+p_2p_3$ is the second elementary symmetric function of $\bfp=(p_1,p_2,p_3)$. It follows that $\grs_2$ is odd. 
The following theorem is implicit in \cite{Kru97}, but we give its simple proof here for completeness.

\begin{theorem}\label{pwnothomequiv}
The 7-manifolds $M^7_\bfp$ and $M^7_\bfw$ are not homotopy equivalent for any admissible $\bfp$ or $\bfw$.
\end{theorem}

\begin{proof}
These manifolds are distinguished by $\pi_4$. Our manifolds $M^7_\bfw$ are quotients of $S^5\times S^3$ by a free $S^1$-action, whereas, the manifolds $M^7_\bfp$ of \cite{BGM94} are free $S^1$ quotients of $SU(3)$. So from their long exact homotopy sequences we have $\pi_i(M^7_\bfw)\approx \pi_i(S^5\times S^3)$ and $\pi_i(M^7_\bfp)\approx \pi_i(SU(3))$ for all $i>2$. But it is known that $\pi_4(SU(3))\approx 0$ whereas, $\pi_4(S^5\times S^3)\approx \bbz_2$.
\end{proof}

\subsubsection{$M=S^2\times S^3, N=\bbc\bbp^1\times \bbc\bbp^1$}
We have $\cali_N=2$, so there are two cases: $|\bfw|$ is odd impying $l_2=|\bfw|$ and $l_1=2$; and $|\bfw|$ is even with $l_2=\frac{|\bfw|}{2}$ and $l_1=1$. In both cases the smoothness condition $\gcd(l_2,l_1w_i)=1$ is satisfied. The $E_2$ term of the Leray-Serre spectral sequence of the top fibration of diagram (\ref{orbifibrationexactseq}) is 
$$E^{p,q}_2=H^p(\mathsf{B}S^1,H^q(S^2\times S^3\times S^3_\bfw,\bbz))\approx \bbz[s]\otimes\bbz[\gra]/(\gra^2)\otimes \grL[\grb,\grg],$$ 
which by the Leray-Serre Theorem converges to $H^{p+q}(M_{l_1,l_2,\bfw},\bbz)$. Here $\gra$ is a 2-class and $\grb,\grg$ are 3-classes. From the bottom fibration in Diagram (\ref{orbifibrationexactseq}) we have $d_2(\grb)=\gra\otimes s_1$ and $d_4(\grg)=w_1w_2s^2_2$. From the commutativity of diagram \eqref{orbifibrationexactseq} we have $d_2(\grb)=l_2s$ and $d_4(\grg_\bfw)=w_1w_2l_1^2s^2$ which gives $E_4^{4,0}\approx \bbz_{w_1w_2l_1^2}$, $E_4^{0,3}\approx \bbz$,  $E_4^{2,2}\approx \bbz_{l_2}$,  and $E_\infty^{0,3}=0$. Then using Poincar\'e duality and universal coefficients we obtain

\begin{proposition}\label{MNprop}
In this case $M^7_{l_1,l_2,\bfw}$ with either $(l_1,l_2)=(2,|\bfw|)$ or $(1,\frac{|\bfw|}{2})$ has the cohomology ring given by
$$H^*(M^7_{l_1,l_2,\bfw},\bbz)=\bbz[x,y,u,z]/(x^2,l_2xy,w_1w_2l_1^2y^2,z^2,u^2,zu,zx,ux,uy)$$
where $x,y$ are 2-classes, and $z,u$ are 5-classes. 
\end{proposition}

\subsubsection{$M=k(S^2\times S^3), N=N_k=\bbc\bbp^2\# k\overline{\bbc\bbp}^2, k=1,\ldots,8$}
Let $\cals_k$ denote the total space of the principal $S^1$-bundle over $N_k$ corresponding to the anticanonical line bundle $K^{-1}$ on $N_k$. By a well-known result of Smale $\cals_k$ is diffeomorphic to the $k$-fold connected sum $k(S^2\times S^3)$. We consider the join $\cals_k\star_{1,|\bfw|}S^3_\bfw$. The case $\bfw=(1,1)$ was studied in \cite{BG00a} where it is shown to have a Sasaki-Einstein metric when $3\leq k\leq 8$. Moreover, in this case we have determined the integral cohomology ring (see Theorem 5.4 of \cite{BG00a}). Here we generalize this result.

\begin{theorem}\label{Sk7man}
The integral cohomology ring of the 7-manifolds $M^7_{k,\bfw}=\cals_k\star_{1,|\bfw|} S^3_\bfw$ is given by
$$H^q(M^7_{k,\bfw},\bbz)\approx \begin{cases}
                    \bbz & \text{if $q=0,7$;} \\
                    \bbz^{k+1} & \text{if $q=2,5$;} \\
                    \bbz^k_{|\bfw|}\times \bbz_{w_1w_2} & \text{if $q=4;$} \\
                     0 & \text{if otherwise}, 
                     \end{cases}$$
with the ring relations determined by $\gra_i\cup \gra_j=0, w_1w_2s^2=0, |\bfw|\gra_i\cup
s=0,$ where $\gra_i,s$ are the $k+1$ two classes with $i=1,\cdots k.$
\end{theorem}

\begin{proof}
As before the $E_2$ term of the Leray-Serre spectral sequence of the top fibration of diagram (\ref{orbifibrationexactseq}) is 
$$E^{p,q}_2=H^p(\mathsf{B}S^1,H^q(\cals_k\times S^3_\bfw,\bbz))\approx \bbz[s]\otimes\prod_i\grL[\gra_i,\grb_i,\grg]/ \gI,$$
where $\gra_i,\grb_j,\grg$ have degrees $2,3,3$, respectively, and $\gI$ is the ideal generated by the relations $\gra_i\cup \grb_i=\gra_j\cup \grb_j,\gra_i\cup\gra_j=\grb_i\cup\grb_j=0$ for all $i,j$, $\gra_i\cup\grb_j=0$ for $i\neq j$ and $\grg^2=0$.

Consider the lower product fibration of diagram (\ref{orbifibrationexactseq}). As in the previous case the first non-vanishing differential of the second factor is $d_4$, and as in that case $d_4(\grg)=w_1w_2s^2_2$. For the first factor we know from Smale's classification of simply connected spin 5-manifolds that $\cals_k$ is diffeomorphic to the $k$-fold connected sum $k(S^2\times S^3)$. Moreover, since $N=\bbc\bbp^2\# k\overline{\bbc\bbp}^2$, the first factor fibration is
$$k(S^2\times S^3)\ra{1.8} \bbc\bbp^2\# k\overline{\bbc\bbp}^2\ra{1.8} \mathsf{B}S^1.$$
Here the first non-vanishing differential is $d_2(\grb_i)=\gra_i\otimes s$.
Again from the commutativity of diagram (\ref{orbifibrationexactseq}) for the top fibration we have $d_2(\grb_i)=|\bfw|\gra_i\otimes s$ at the $E_2$ level and $d_4(\grg)=w_1w_2s^2$ at the $E_4$ level. One easily sees that the $k+1$ $2$-classes $\gra_i\in E_2^{2,0}$ and $s\in E_2^{0,2}$ live to $E_\infty$ and there is no torsion in degree $2$. Moreover, there is nothing in degree $1$, and the $3$-classes $\grb_i\in E_2^{3,0}$ and $\grg\in E_4^{3,0}$ die, so there is nothing in degree $3$. However, there is torsion in degree $4$, namely $\bbz_{|\bfw|}^k\times \bbz_{w_1w_2}$. The remainder follows from Poincar\'e duality and dimensional considerations.
\end{proof}

This generalizes Theorem 5.4 of \cite{BG00a} where the case $\bfw=(1,1)$ is treated.

\begin{remark} 
Since $|\bfw|$ and $w_1w_2$ are relatively prime, $H^4(M^7_{k,\bfw},\bbz)\approx \bbz^{k-1}_{|\bfw|}\times \bbz_{w_1w_2|\bfw|}$. We can ask the question: when can $M^7_{k,\bfw}$ and $M^7_{k',\bfw'}$ have isomorphic cohomology rings? It is interesting and not difficult to see that there is only one possibility, namely $M^7_{1,(3,2)}$ and $M^7_{1,(5,1)}$ in which case $H^4\approx \bbz_{30}$. 
\end{remark}

\section{Admissible KE constructions}\label{KE}

We now pick up the thread from Section \ref{adm-basics} and describe the construction (see also
\cite{ACGT08}) of admissible K\"ahler metrics on $S_n$ (in fact, more generally on log pairs $(S_n, \Delta)$).
Consider the circle action
on $S_n$ induced by the natural circle action on $L_n$. It extends to a holomorphic
$\mathbb{C}^*$ action. The open and dense set ${S_n}_0\subset S_n$ of stable points with respect to the
latter action has the structure of a principal circle bundle over the stable quotient.
The hermitian norm on the fibers induces, via a Legendre transform, a function
$\gz:{S_n}_0\rightarrow (-1,1)$ whose extension to $S_n$ consists of the critical manifolds
$\gz^{-1}(1)=P(\BOne\oplus 0)$ and $\gz^{-1}(-1)=P(0 \oplus L_n)$.
Letting $\theta$ be a connection one form for the Hermitian metric on ${S_n}_0$, with curvature
$d\theta = \omega_{N_n}$, an admissible K\"ahler metric and form are
given up to scale by the respective formulas
\begin{equation}\label{g}
g=\frac{1+r\gz}{r}g_{N_n}+\frac {d\gz^2}
{\Theta (\gz)}+\Theta (\gz)\theta^2,\quad
\omega = \frac{1+r\gz}{r}\omega_{N_n} + d\gz \wedge
\theta,
\end{equation}
valid on ${S_n}_0$. Here $\Theta$ is a smooth function with domain containing
$(-1,1)$ and $r$, is a real number of the same sign as
$g_{N_n}$ and satisfying $0 < |r| < 1$. The complex structure yielding this
K\"ahler structure is given by the pullback of the base complex structure
along with the requirement $Jd\gz = \Theta \theta$. The function $\gz$ is hamiltonian
with $K= J\,grad\, \gz$ a Killing vector field. In fact, $\gz$ is the moment 
map on $S_n$ for the circle action, decomposing $S_n$ into 
the free orbits ${S_n}_0 = \gz^{-1}((-1,1))$ and the special orbits 
$D_1= \gz^{-1}(1)$ and $D_2=\gz^{-1}(-1)$.
Finally, $\theta$ satisfies
$\theta(K)=1$.

\begin{remark}\label{hamiltonian2form}
Note that
$$\phi := \frac{-(1+r \gz)}{r^2} \omega_{N_n} + \gz d\gz \wedge \theta$$ is a Hamiltonian $2$-form of order one.
\end{remark}

\bigskip

We can now interpret $g$ as a metric on the log pair $(S_n,\Delta)$ with
$$\grD= (1-1/m_1)D_1+(1-1/m_2)D_2$$ if
$\Theta$ satisfies the positivity and boundary
conditions
\begin{equation}
\label{positivity}
\begin{array}{l}
\Theta(\gz) > 0, \quad -1 < \gz <1,\\
\\
 \Theta(\pm 1) = 0,\\
 \\
 \Theta'(-1) = 2/m_2\quad \quad \Theta'(1)=-2/m_1.
\end{array}
\end{equation}

\begin{remark}
This construction is based on the symplectic viewpoint where different choices of $\Theta$ yields different complex structures all compatible with the same fixed symplectic form $\omega$. However, for each $\Theta$ there is an $S^1$-equivariant diffeomorphism pulling back $J$ to the original fixed complex structure on $S_n$ in such a way that the K\"ahler form of the new K\"ahler metric is in the same cohomology class as $\omega$ \cite{ACGT08}. Therefore, with all else fixed, we may view the set of the functions $\Theta$ satisfying \eqref{positivity} as parametrizing a family of K\"ahler metrics within the same K\"ahler class of $(S_n,\Delta)$.
\end{remark}

\bigskip

The K\"ahler class $\Omega_{\mathbf r} = [\omega]$ of an admissible metric is also called
{\it admissible} and is uniquely determined by the parameter
$r$, once the data associated with $S_n$ (i.e.
$d_N$, $s_{N_n}$, $g_{N_n}$ etc.) is fixed. In fact,
\begin{equation}\label{admKahclass}
\Omega_{\mathbf r}  = [\omega_{N_n}]/r + 2 \pi PD[D_1+D_2],
\end{equation}
where $PD$ denotes the Poincar\'e dual.
The number $r$,
together with the data associated with $S_n$ will be called {\it admissible data}.

Define a function $F(\gz)$ by the formula $\Theta(\gz)=F(\gz)/\gp(\gz)$, where
$\gp(\gz) =(1 + r \gz)^{d_{N}}$.
Since $\gp(\gz)$ is positive for $-1\leq \gz \leq1$, conditions
\eqref{positivity}
are equivalent to the following conditions on $F(\gz)$:
\begin{equation}
\label{positivityF}
\begin{array}{l}
F(\gz) > 0, \quad -1 < \gz <1,\\
\\
F(\pm 1) = 0,\\
\\
F'(- 1) = 2\gp(-1)/m_2 \quad \quad F'( 1) =-2\gp(1)/m_1.
\end{array}
\end{equation}

\subsection{The Einstein Conditions}
A K\"ahler metric is KE if and only if
$$\rho - \lambda \omega=0$$ for some constant $\lambda$. From \cite{ApCaGa06} we have that the Ricci form of an admissible metric  given by \eqref{g} equals
\begin{equation}\label{rho}
\rho =
\rho_{N} - \frac{1}{2} d d^c \log F =s_{N_n}\omega_{N_n} - 
\frac{1}{2}\frac{F'(\gz)}{\gp(\gz)}  \omega_{N_n}
-\frac{1}{2}\Bigl(\frac{F'(\gz)}{\gp(\gz)}\Bigr)'(\gz) d\gz \wedge \theta.
\end{equation}
Thus the KE 
condition is equivalent the 
ODE
\begin{equation} \label{KEodes}
\frac{F'(\gz)}{\gp(\gz)}  = 2s_{N_n} - 2 \lambda 
(\gz + 1/r).
\end{equation}
Now \eqref{positivityF} implies the necessary conditions
$$ 
\begin{array}{ccl}
s_{N_n} -  \lambda 
(-1 + 1/r)& = & 1/m_2\\
\\
s_{N_n}-  \lambda 
(1 + 1/r)&=& -1/m_1,
\end{array}
$$
which are equivalent to
\begin{equation}\label{fano}
\begin{array}{ccl}
2\lambda& = & 1/m_2+1/m_1\\
\\
2s_{N_n}r &=& (1+r)/m_2 + (1-r)/m_1.
\end{array}
\end{equation}

Since $s_{N_n}r >0$ we see that the base manifold $N$ (not surprisingly) must have positive scalar curvature.
If  \eqref{fano} is satisfied, then
\eqref{KEodes} is equivalent to the ODE:
\begin{equation} \label{KEode}
\frac{F'(\gz)}{\gp(\gz)}  = (1-\gz)/m_2 -(1+\gz)/m_1
\end{equation}

Now it is easy to see that for a solution satisfying \eqref{positivityF} to exist we
need
\begin{equation}\label{KEintegral}
\int_{-1}^1 \left((1-\gz)/m_2 -(1+\gz)/m_1\right) {\gp(\gz)} d\gz = 0.
\end{equation}
On the other hand, if this is satisfied
\begin{equation}\label{KEmetricF}
F(\gz) := \int_{-1}^\gz \left((1-t)/m_2 -(1+t)/m_1\right) {\gp(t)} dt
\end{equation}
would yield a solution of \eqref{KEodes} satisfying all the conditions of \eqref{positivityF}. 
Setting $s_{N_n}=\cali_N/n$ in the second equation of \eqref{fano} we have the following result.
\begin{proposition}\label{KEprop} 
Given admissible data and a choice of $m_1,m_2$ as above the admissible metric  \eqref{g}, 
with $\Theta(\gz) = \frac{F(\gz)}{\gp(t)}$ and $F(\gz)$ given by \eqref{KEmetricF},
is KE iff  
$$2r\cali_N/n = (1+r)/m_2 + (1-r)/m_1$$
and \eqref{KEintegral} are both satisfied.
\end{proposition}

\begin{lemma}
For the log pair $(S_n,\Delta)$ with
$$\grD= (1-1/m_1)D_1+(1-1/m_2)D_2$$ 
the orbifold Chern class equals
$$c_1^{orb}(S_n,\Delta) = c_1(N) + \frac{1}{m_1}PD(D_1)+\frac{1}{m_2}PD(D_2),$$
where $c_1(N)$ is viewed as a pull-back
\end{lemma}
\begin{proof}
The usual argument gives that
$$c_1^{orb}(S_n,\Delta) = c_1(S_n) +  (\frac{1}{m_1}-1)PD(D_1)+(\frac{1}{m_2}-1)PD(D_2)$$
and the lemma now follows from the fact that
$$c_1(S_n) = c_1(N) + PD(D_1)+PD(D_2).$$
One can verify the last fact by using the explicit Ricci form above for some convenient choice of admissible metric (e.g. take $F(\gz)=(1-\gz^2)\gp(\gz)$) in the case $m_1=m_2=1$, but it should also follow from general principles.
\end{proof}

In order to prove Theorem \ref{admjoinse}, we now assume that $(S_n,\Delta)$ arises as the quotient of the flow of a quasi-regular Reeb vector field, $\xi_\bfv$,  as in Theorem \ref{preSE}. Since $c_1(\cald)=0$, the K\"ahler class of the quotient metric must be a multiple of $c_1^{orb}(S_n,\Delta)$.
Since $2\pi c_1(N) = s_{N_n}[\omega_{N_n}] = \cali_N [\omega_{N_n}] /n$ and $2\pi(PD(D_1)-PD(D_2))=[\omega_{N_n}]$ (see e.g. Section 1.3 in \cite{ACGT08}) we see that
$$4\pi c_1^{orb}(S_n,\Delta) = (2\cali_N/n + 1/m_1-1/m_2)[\omega_{N_n}] + (1/m_1+1/m_2)2\pi(PD(D_1)+PD(D_2))$$
and thus $\Omega_{\mathbf r}$ is a multiple of $c_1^{orb}(S_n,\Delta)$ precisely when the first condition of Propostion \ref{KEprop}, namely
$2r\cali_N/n = (1+r)/m_2 + (1-r)/m_1$, is satisfied. Using the formulas for $n$, $m_1$ and $m_2$ in Theorem \ref{preSE} we get
$$ r = \frac{w_1v_2-w_2v_1}{w_1v_2+w_2v_1}.$$

\begin{remark}
As can be seen in our paper \cite{BoTo14a}, the factor relating the K\"ahler class of the quotient metric and the K\"ahler class $\Omega_r$ only depends on the initial choice of $l_2$ (to be precise the factor is $l_2/4\pi$), so for simplicity we will simply ignore it from here on out. 
\end{remark}

Canceling out common factors, Proposition \ref{KEprop} implies that (up to isotopy) the Sasaki structure associated to
$\xi_\bfv$ is  $\eta$-Einstein (and thus, up to transverse homothety, SE)
iff
\begin{equation}\label{KEintegral2}
\int_{-1}^1 \left((v_1-v_2)-(v_1+v_2)\gz) \right)((w_1v_2+w_2v_1)+ (w_1v_2-w_2v_1)\gz)^{d_N}d\gz = 0.
\end{equation}
For convenience we set $w_2/w_1=t$ and $v_2/v_1=c$. For the admissible set-up to make sense we assume $c\neq t$. We also assume $0<t< 1$ (i.e. $w_1>w_2$). Now equation \eqref{KEintegral2} is equivalent to
\begin{equation}\label{KEintegral3}
\int_{-1}^1 \left((1-c)-(1+c)\gz) \right)((c+t)+ (c-t)\gz)^{d_N}d\gz = 0.
\end{equation}
Let $f(c)$ denote the left hand side of \eqref{KEintegral3} and assume $t\in (0,1)\cap{\mathbb Q}$ is fixed. Now it is easy to check that
$$f(t) >0\quad\quad\text{and}\quad\quad \lim_{c\rightarrow +\infty}f(c) =-\infty.$$
Thus $\exists c_{t} \in (t,+\infty)$ such that \eqref{KEintegral3} is solved. 
Usually this $c_t$ will be irrational. As discussed in Remark 5.4 of \cite{BoTo13}, this will correspond to an irregular SE structure. In order to understand this we give the admissible construction on the Sasaki level.

Let $M_0$ denote the subspace of $M_{l_1,l_2,\bfw}$ where $\xi_\bfv$ acts freely. Note that $M_0$ is independent of $\bfv$ and is the total space of a circle bundle over $S_{n0}$. Moreover, for any $\bfv$ we have
\begin{equation}\label{M0}
M_{l_1,l_2,\bfw}/M_0=\pi_\bfv^*D_1\bigsqcup \pi_\bfv^*D_2.
\end{equation}
We can now pullback the admissible K\"ahler class given by \eqref{admKahclass} to give an admissible transverse K\"ahler class $\grO_r^T\in H^{1,1}_B(\calf_{\xi_\bfv})$, as well as the admissible data defined in Equations \eqref{positivity} and \eqref{g} to give the admissible Sasakian data in the quasi-regular case. Explicitly we have the transverse K\"ahler metric and form
\begin{equation}\label{gT}
g^T=\frac{1+r\gtz}{r}\pi_\bfv^*g_{N_n}+\frac {d\gtz^2}
{\Theta (\gtz)}+\Theta (\gtz)(\pi_\bfv^*\theta)^2,\quad
d\eta_\bfv = \frac{1+r\gtz}{r}\pi_\bfv^*\omega_{N_n} + d\gtz \wedge \pi_\bfv^*\theta,
\end{equation}
where $\gtz:M_{l_1,l_2,\bfw}\ra{1.6} [-1,1]$ is the moment map of the lifted circle action of the moment map $\gz$. Furthermore, $\Theta(\gtz)$ satisfies its previous conditions, and it is important to realize that the only dependence of the admissible Sasakian data on $\bfv$ is through the boundary conditions 
$$ \Theta'(-1) = 2/m_2,\quad \quad \Theta'(1)=-2/m_1, \qquad m_i=v_im.$$
We then get a Sasaki metric in the usual way, namely $g_\bfv=g^T+\eta_\bfv\otimes \eta_\bfv$ together with the full Sasakian structure $\cals_\bfv=(\xi_\bfv,\eta_\bfv,\Phi_\bfv,g_v)$. Although this construction was done for a pair of relatively prime positive integers $v_1,v_2$ we see that by continuity all the data in the construction makes perfect sense on $M_{l_1,l_2,\bfw}$ for any pair of positive real numbers $v_1,v_2$. This defines the admissible Sasaki data in the case of irregular Sasakian structures. Thus, an irregular solution to Equation \eqref{KEintegral3} gives a positive $\eta$-Einstein metric on the Sasaki manifold and hence an SE metrics by a transverse homothety.

\begin{example}
Although the majority of the SE structures obtained in this paper are irregular,
we can, however, produce many quasi-regular SE cases as follows: Set $c=kt$. Then 
\eqref{KEintegral3} is equivalent with
\begin{equation}\label{KEintegral4}
\int_{-1}^1 \left((1-kt)-(1+kt)\gz) \right)((k+1)+ (k-1)\gz)^{d_N}d\gz = 0.
\end{equation}
or
\begin{equation}\label{KEintegral5}
t= \frac{\int_{-1}^1 \left(1-\gz \right)((k+1)+ (k-1)\gz)^{d_N}d\gz}{k\int_{-1}^1 \left(1+\gz \right)((k+1)+ (k-1)\gz)^{d_N}d\gz}.
\end{equation}
\begin{lemma}
For $k>1$,
$$0<\int_{-1}^1 \left(1-\gz \right)((k+1)+ (k-1)\gz)^{d_N}d\gz<\int_{-1}^1 \left(1+\gz \right)((k+1)+ (k-1)\gz)^{d_N}d\gz.$$
\end{lemma}
\begin{proof}
The first inequality is obvious and the next is equivalent to 
$$\int_{-1}^{1} \gz ((k+1)+ (k-1)\gz)^{d_N}d\gz >0.$$
By integrating, this in turn is equivalent to
$$-d_{N_n} + (2+d_{N_n})k - (2+d_{N_n})k^{d_{N_n}+1} + d_{N_n}k^{d_{N_n}+2} >0.$$
Setting $p(k) = -d_{N_n} + (2+d_{N_n})k - (2+d_{N_n})k^{d_{N_n}+1} + d_{N_n}k^{d_{N_n}+2}$ we  observe that $p(1)=p'(1)=0$ while $p\,''(k) > 0$ for all $k>1$. Thus $p(k)>0$ for all $k>1$ and hence the inequality holds.
 \end{proof}
Now it follows that for any given $k\in (1,+\infty)\cap {\mathbb Q}$, $\exists t \in (0,1)\cap {\mathbb Q}$ (determined by \eqref{KEintegral5}) such that if the co-prime integers $w_1$ and $w_2$ are such that $w_2/w_1=t$ and then co-prime integers $v_1$ and $v_2$ are such that $v_2/v_1= kt$ (and $l_1$ and $l_2$ are chosen according to Lemma \ref{c10}) then the ray determined by $(v_1,v_2)$ in the $\bfw$-Sasaki cone contains a quasi-regular SE structure.

Note that when $d_N=1$, equation \eqref{KEintegral5} is
$$t=\frac{2+k}{k(1+2k)}.$$
The quasi-regular SE solutions \cite{GMSW04a} for $Y^{p,q}$, as described in Example \ref{Ypq}, are recovered by choosing
$$k= \frac{q+\sqrt{4p^2-3q^2}}{2(p-q)},$$
assuming, as is prescribed by e.g. Theorem 11.4.5 in \cite{BG05},
that $4p^2-3q^2=n^2$, for some $n \in \bbz$. 
Note that, conversely for $k=a/b$ with co-prime $a>b \in \bbz^+$, we have
$p=a b + a^2+b^2$ and $q=a^2-b^2$.
\end{example}

In general, the following result follows from Theorem \ref{admjoinse}.
\begin{proposition}\label{ypqse}
The Reeb vector field of the unique Sasaki-Einstein metric of $Y^{p,q}$ lies in the $\bfw$-Sasaki cone with $\bfw$ determined by Equations (\ref{pqw}).
\end{proposition}

\subsection{Sasaki-Ricci solitons and extremal Sasaki metrics}
Although our main focus in this paper has been on SE structures, it is natural to note that
generalizing \eqref{KEodes} to
\begin{equation} \label{KRSode}
\frac{F'(\gz)}{\gp(\gz)}   - a \frac{F'(\gz)}{\gp(\gz)}= 2s_{N_n} - 2 \lambda 
(\gz + 1/r),
\end{equation}
where $a\in\bbr$ is some constant, corresponds to generalizing
the KE equation $\rho -\lambda \omega=0$ to the K\"ahler Ricci soliton (KRS) equation
$$\rho -\lambda \omega = \call_V\omega,$$
with $V = \frac{a}{2} grad_g \gz$.
By following e.g. Section 3 in \cite{ACGT08b} and adapting it to our more general endpoint conditions \eqref{positivityF} (but letting $d_0=d_\infty=0$), it is now straightforward and completely standard
to verify that Proposition \ref{KEprop} generalizes with ``KE'' replaced by ``KRS'',
\eqref{KEintegral} replaced by
\begin{equation}\label{KRSintegral}
\int_{-1}^1  e^{-a\, \gz}\left((1-\gz)/m_2 -(1+\gz)/m_1\right) {\gp(\gz)} d\gz = 0,
\end{equation}
and \eqref{KEmetricF} replaced by
\begin{equation}\label{KRSmetricF}
F(\gz) :=e^{a\,\gz} \int_{-1}^\gz e^{-a\,t} \left((1-t)/m_2 -(1+t)/m_1\right) {\gp(t)} dt.
\end{equation}

Moreover, equation \eqref{KRSintegral} can always be solved for some $a\in \bbr$.
Thus we realize that  (up to isotopy) the Sasaki structure associated to every single ray, $\xi_\bfv$, in our $\bfw$-Sasaki cone is a Sasaki-Ricci soliton
(as defined in \cite{FOW06}). We mention also that Sasaki-Ricci solitons on toric 5-manifolds were studied in \cite{LeTo13}.

Our set-up (starting from a join construction) allows for cases where no regular ray in the $\bfw$-Sasaki cone exists.
If, however, the given $\bfw$-Sasaki cone does admit a regular ray, then the transverse K\"ahler structure is a smooth K\"ahler Ricci soliton and the existence of an SE metric in some  ray of the Sasaki cone is predicted by the work of \cite{MaNa13}.

Another generalization of \eqref{KEodes} would be to require that
\begin{equation}\label{extremal1}
F''(\gz) = (1+r \gz)^{d_N-1} P(\gz),
\end{equation}
where $P(\gz)$ is a polynomial of degree $2$ satisfying that
\begin{equation}\label{extremal2}
P(-1/r) = 2 d_N s_{N_n} r.
\end{equation}
It is well known that this corresponds to extremal K\"ahler metrics (see e.g. \cite{ACGT08}).
Moreover, similarly to the smooth case, one easily sees that \eqref{extremal1} with \eqref{extremal2} has a unique solution $F(\gz)$ satisfying the endpoint conditions of \eqref{positivityF}. 
Finally, since $N$ is here a {\em positive} K\"ahler-Einstein metric, this polynomial $F(\gz)$ also satisfies the positivity condition of \eqref{positivityF} by the standard root-counting argument introduced by Hwang \cite{Hwa94} and Guan \cite{Gua95}. Thus (up to isotopy) the Sasakian structure associated to every single ray, $\xi_\bfv$, in our $\bfw$-Sasaki cone is extremal.

\def\cprime{$'$} \def\cprime{$'$} \def\cprime{$'$} \def\cprime{$'$}
  \def\cprime{$'$} \def\cprime{$'$} \def\cprime{$'$} \def\cprime{$'$}
  \def\cdprime{$''$} \def\cprime{$'$} \def\cprime{$'$} \def\cprime{$'$}
  \def\cprime{$'$}
\providecommand{\bysame}{\leavevmode\hbox to3em{\hrulefill}\thinspace}
\providecommand{\MR}{\relax\ifhmode\unskip\space\fi MR }

\providecommand{\MRhref}[2]{%
  \href{http://www.ams.org/mathscinet-getitem?mr=#1}{#2}
}
\providecommand{\href}[2]{#2}


\begin{thebibliography}{ACGTF08b}

\bibitem[ACG06]{ApCaGa06}
Vestislav Apostolov, David M.~J. Calderbank, and Paul Gauduchon,
  \emph{Hamiltonian 2-forms in {K}\"ahler geometry. {I}. {G}eneral theory}, J.
  Differential Geom. \textbf{73} (2006), no.~3, 359--412. \MR{2228318
  (2007b:53149)}

\bibitem[ACGTF04]{ACGT04}
V.~Apostolov, D.~M.~J. Calderbank, P.~Gauduchon, and C.~W.
  T{\o}nnesen-Friedman, \emph{Hamiltonian $2$-forms in {K}\"ahler geometry.
  {II}. {G}lobal classification}, J. Differential Geom. \textbf{68} (2004),
  no.~2, 277--345. \MR{2144249}

\bibitem[ACGTF08a]{ACGT08c}
Vestislav Apostolov, David M.~J. Calderbank, Paul Gauduchon, and Christina~W.
  T{\o}nnesen-Friedman, \emph{Extremal {K}\"ahler metrics on ruled manifolds
  and stability}, Ast\'erisque (2008), no.~322, 93--150, G{\'e}om{\'e}trie
  diff{\'e}rentielle, physique math{\'e}matique, math{\'e}matiques et
  soci{\'e}t{\'e}. II. \MR{2521655 (2010h:32029)}

\bibitem[ACGTF08b]{ACGT08}
\bysame, \emph{Hamiltonian 2-forms in {K}\"ahler geometry. {III}. {E}xtremal
  metrics and stability}, Invent. Math. \textbf{173} (2008), no.~3, 547--601.
  \MR{MR2425136 (2009m:32043)}

\bibitem[ACGTF08c]{ACGT08b}
\bysame, \emph{Hamiltonian 2-forms in {K}\"ahler geometry. {IV}. {W}eakly
  {B}ochner-flat {K}\"ahler manifolds}, Comm. Anal. Geom. \textbf{16} (2008),
  no.~1, 91--126. \MR{2411469 (2010c:32043)}

\bibitem[BG99]{BG99}
C.~P. Boyer and K.~Galicki, \emph{3-{S}asakian manifolds}, Surveys in
  differential geometry: essays on Einstein manifolds, Surv. Differ. Geom., VI,
  Int. Press, Boston, MA, 1999, pp.~123--184. \MR{2001m:53076}

\bibitem[BG00]{BG00a}
\bysame, \emph{On {S}asakian-{E}instein geometry}, Internat. J. Math.
  \textbf{11} (2000), no.~7, 873--909. \MR{2001k:53081}

\bibitem[BG08]{BG05}
Charles~P. Boyer and Krzysztof Galicki, \emph{Sasakian geometry}, Oxford
  Mathematical Monographs, Oxford University Press, Oxford, 2008. \MR{MR2382957
  (2009c:53058)}

\bibitem[BGM94]{BGM94}
Charles~P. Boyer, Krzysztof Galicki, and Benjamin~M. Mann, \emph{The geometry
  and topology of {$3$}-{S}asakian manifolds}, J. Reine Angew. Math.
  \textbf{455} (1994), 183--220. \MR{MR1293878 (96e:53057)}

\bibitem[BGO07]{BGO06}
Charles~P. Boyer, Krzysztof Galicki, and Liviu Ornea, \emph{Constructions in
  {S}asakian geometry}, Math. Z. \textbf{257} (2007), no.~4, 907--924.
  \MR{MR2342558 (2008m:53103)}

\bibitem[Boy11]{Boy11}
Charles~P. Boyer, \emph{Completely integrable contact {H}amiltonian systems and
  toric contact structures on {$S^2\times S^3$}}, SIGMA Symmetry Integrability
  Geom. Methods Appl. \textbf{7} (2011), Paper 058, 22. \MR{2861218}

\bibitem[BP14]{BoPa10}
Charles~P. Boyer and Justin Pati, \emph{On the equivalence problem for toric
  contact structures on ${S}^3$-bundles over ${S}^2$}, Pac. Jour. of Math.
  \textbf{267} (2014), no.~2, 277--324.

\bibitem[BTF13a]{BoTo11}
Charles~P. Boyer and Christina~W. T{\o}nnesen-Friedman, \emph{Extremal
  {S}asakian geometry on {$T^2\times S^3$} and related manifolds}, Compos.
  Math. \textbf{149} (2013), no.~8, 1431--1456. \MR{3103072}

\bibitem[BTF13b]{BoTo12b}
\bysame, \emph{Sasakian manifolds with perfect fundamental groups}, Afr.
  Diaspora J. Math. \textbf{14} (2013), no.~2, 98--117. \MR{3093238}

\bibitem[BTF14a]{BoTo13}
\bysame, \emph{Extremal {S}asakian geometry on ${ S}^3$-bundles over {R}iemann
  surfaces}, Int. Math. Res. Not. IMRN (2014), doi:10.1093/139.

\bibitem[BTF14b]{BoTo14a}
\bysame, \emph{The {S}asaki join, {H}amiltonian 2-forms, and constant scalar
  curvature}, preprint; arXiv:1402.2546 Math.DG (2014).

\bibitem[BW58]{BoWa}
W.~M. Boothby and H.~C. Wang, \emph{On contact manifolds}, Ann. of Math. (2)
  \textbf{68} (1958), 721--734. \MR{22 \#3015}

\bibitem[CLPP05]{CLPP05}
M.~Cveti{\v{c}}, H.~L{\"u}, D.~N. Page, and C.~N. Pope, \emph{New
  {E}instein-{S}asaki spaces in five and higher dimensions}, Phys. Rev. Lett.
  \textbf{95} (2005), no.~7, 071101, 4. \MR{2167018}

\bibitem[CS12]{CoSz12}
Tristan Collins and Gabor Sz\'ekelyhidi, \emph{K-semistability for irregular
  {S}asakian manifolds}, preprint; arXiv:math.DG/1204.2230 (2012).

\bibitem[Esc05]{Esc05}
C.~M. Escher, \emph{A diffeomorphism classification of generalized {W}itten
  manifolds}, Geom. Dedicata \textbf{115} (2005), 79--120. \MR{2180043
  (2006i:57058)}

\bibitem[FOW09]{FOW06}
Akito Futaki, Hajime Ono, and Guofang Wang, \emph{Transverse {K}\"ahler
  geometry of {S}asaki manifolds and toric {S}asaki-{E}instein manifolds}, J.
  Differential Geom. \textbf{83} (2009), no.~3, 585--635. \MR{MR2581358}

\bibitem[GHP03]{GHP03}
G.~W. Gibbons, S.~A. Hartnoll, and C.~N. Pope, \emph{Bohm and
  {E}instein-{S}asaki metrics, black holes, and cosmological event horizons},
  Phys. Rev. D (3) \textbf{67} (2003), no.~8, 084024, 24. \MR{1995313
  (2004i:83070)}

\bibitem[GMSW04a]{GMSW04b}
J.~P. Gauntlett, D.~Martelli, J.~Sparks, and D.~Waldram, \emph{A new infinite
  class of {S}asaki-{E}instein manifolds}, Adv. Theor. Math. Phys. \textbf{8}
  (2004), no.~6, 987--1000. \MR{2194373}

\bibitem[GMSW04b]{GMSW04a}
\bysame, \emph{Sasaki-{E}instein metrics on {$S^2\times S^3$}}, Adv. Theor.
  Math. Phys. \textbf{8} (2004), no.~4, 711--734. \MR{2141499}

\bibitem[Gua95]{Gua95}
Daniel Guan, \emph{Existence of extremal metrics on compact almost homogeneous
  {K}\"ahler manifolds with two ends}, Trans. Amer. Math. Soc. \textbf{347}
  (1995), no.~6, 2255--2262. \MR{1285992 (96a:58059)}

\bibitem[Hae84]{Hae84}
A.~Haefliger, \emph{Groupo\"\i des d'holonomie et classifiants}, Ast\'erisque
  (1984), no.~116, 70--97, Transversal structure of foliations (Toulouse,
  1982). \MR{86c:57026a}

\bibitem[HS12]{HeSu12b}
Weiyong He and Song Sun, \emph{The generalized {F}rankel conjecture in {S}asaki
  geometry}, preprint; arXiv:math.DG/1209.4026 (2012).

\bibitem[Hwa94]{Hwa94}
Andrew~D. Hwang, \emph{On existence of {K}\"ahler metrics with constant scalar
  curvature}, Osaka J. Math. \textbf{31} (1994), no.~3, 561--595. \MR{1309403
  (96a:53061)}

\bibitem[KL05]{KrLu05}
Matthias Kreck and Wolfgang L{\"u}ck, \emph{The {N}ovikov conjecture},
  Oberwolfach Seminars, vol.~33, Birkh\"auser Verlag, Basel, 2005, Geometry and
  algebra. \MR{2117411 (2005i:19003)}

\bibitem[KO73]{KoOc73}
S.~Kobayashi and T.~Ochiai, \emph{Characterizations of complex projective
  spaces and hyperquadrics}, J. Math. Kyoto Univ. \textbf{13} (1973), 31--47.
  \MR{47 \#5293}

\bibitem[Kru97]{Kru97}
B.~Kruggel, \emph{A homotopy classification of certain {$7$}-manifolds}, Trans.
  Amer. Math. Soc. \textbf{349} (1997), no.~7, 2827--2843. \MR{97m:55012}

\bibitem[Kru05]{Kru05}
\bysame, \emph{Homeomorphism and diffeomorphism classification of {E}schenburg
  spaces}, Q. J. Math. \textbf{56} (2005), no.~4, 553--577. \MR{MR2182466
  (2006h:53045)}

\bibitem[KS88]{KS88}
M.~Kreck and S.~Stolz, \emph{A diffeomorphism classification of
  {$7$}-dimensional homogeneous {E}instein manifolds with {${\rm
  SU}(3)\times{\rm SU}(2)\times{\rm U}(1)$}-symmetry}, Ann. of Math. (2)
  \textbf{127} (1988), no.~2, 373--388. \MR{89c:57042}

\bibitem[LTF13]{LeTo13}
Eveline Legendre and Christina~W. T{\o}nnesen-Friedman, \emph{Toric generalized
  {K}\"ahler-{R}icci solitons with {H}amiltonian 2-form}, Math. Z. \textbf{274}
  (2013), no.~3-4, 1177--1209. \MR{3078263}

\bibitem[MN13]{MaNa13}
Toshiki Mabuchi and Yasuhiro Nakagawa, \emph{New examples of
  {S}asaki-{E}instein manifolds}, Tohoku Math. J. (2) \textbf{65} (2013),
  no.~2, 243--252. \MR{3079287}

\bibitem[MS05]{MaSp05b}
D.~Martelli and J.~Sparks, \emph{Toric {S}asaki-{E}instein metrics on
  {$S^2\times S^3$}}, Phys. Lett. B \textbf{621} (2005), no.~1-2, 208--212.
  \MR{2152673}

\bibitem[RT11]{RoTh11}
Julius Ross and Richard Thomas, \emph{Weighted projective embeddings, stability
  of orbifolds, and constant scalar curvature {K}\"ahler metrics}, J.
  Differential Geom. \textbf{88} (2011), no.~1, 109--159. \MR{2819757}

\bibitem[Spa11]{Spa10}
James Sparks, \emph{Sasaki-{E}instein manifolds}, Surveys in differential
  geometry. {V}olume {XVI}. {G}eometry of special holonomy and related topics,
  Surv. Differ. Geom., vol.~16, Int. Press, Somerville, MA, 2011, pp.~265--324.
  \MR{2893680 (2012k:53082)}

\bibitem[WZ90]{WaZi90}
M.~Y. Wang and W.~Ziller, \emph{Einstein metrics on principal torus bundles},
  J. Differential Geom. \textbf{31} (1990), no.~1, 215--248. \MR{91f:53041}

\end{thebibliography}
\end{document}